\documentclass[12pt,a4paper]{article}

\usepackage[left=2cm,right=2cm, bottom = 4cm, top = 4cm]{geometry}
\usepackage{amsfonts,graphicx,amsmath,amssymb,amsthm,bbm,nicefrac,enumerate,comment,mathrsfs, mathtools, todonotes,textcomp}

\usepackage[multiple]{footmisc} 
\usepackage[hyperfootnotes=false]{hyperref}

\hypersetup{colorlinks=true}

\usepackage{cite}
\usepackage[sort, capitalise]{cleveref} %

\crefname{subsection}{Subsection}{Subsections}
\crefname{enumi}{item}{items}
\crefname{equation}{}{}

\theoremstyle{plain}
\newtheorem{theorem}{Theorem}[section]
\newtheorem{lemma}[theorem]{Lemma}
\newtheorem{prop}[theorem]{Proposition}
\newtheorem{cor}[theorem]{Corollary}

\newtheorem{definition}[theorem]{Definition}

\theoremstyle{remark}

\newcommand{\R}{\mathbb{R}}
\newcommand{\N}{\mathbb{N}}

\newcommand{\norm}[1]{\cfadd{euclidean_norm_def} \left\| #1 \right\| }
\newcommand{\Norm}[1]{\cfadd{euclidean_norm_def} \| #1 \| }
\newcommand{\Normm}[1]{\cfadd{euclidean_norm_def} \big\| #1 \big\| }

\newcommand{\normmm}[1]{{\left\vert\kern-0.25ex\left\vert\kern-0.25ex\left\vert #1 
    \right\vert\kern-0.25ex\right\vert\kern-0.25ex\right\vert}} %
\DeclarePairedDelimiter{\newnorm}{\lVert}{\rVert}

\newcommand{\qandq}{\qquad\text{and}\qquad}
\newcommand{\andq}{\text{and}\qquad}

\newcommand{\id}{\operatorname{id}}

\newcommand{\domain}{\operatorname{Domain}}
\newcommand{\codomain}{\operatorname{Codomain}}
\newcommand{\param}{\mathcal{P}}

\newcommand{\realisation}{\cfadd{Def:ANNrealization}\mathcal{R}}

\newcommand{\Flow}{\cfadd{Def:flow_operator} \mathbf{P}}

\newcommand{\inputStruc}{\mathfrak{I}}
\newcommand{\outputStruc}{\mathfrak{O}}
\newcommand{\smallprod}[3]{{\textstyle \prod_{#1 = #2}^{#3}}}
\newcommand{\mycup}{{\textstyle \bigcup}}

\newcommand{\ANNs}{\cfadd{Def:ANN_intro}\mathscr{N}}
\newcommand{\activation}{a}

\newcommand{\functionANN}{\cfadd{Def:ANNrealization}\mathcal{R}_{\activation}}
\newcommand{\paramANN}{\cfadd{Def:ANN_intro}\mathcal{P}}

\newcommand{\lengthANN}{\cfadd{Def:ANN2}\mathcal{L}}
\newcommand{\hidLengthAnn}{\cfadd{Def:ANN2}\mathcal{H}}
\newcommand{\inDimANN}{\cfadd{Def:ANN2}\mathcal{I}}
\newcommand{\compANN}[2]{\cfadd{Definition:ANNcomposition}{#1 \bullet #2}}
\newcommand{\compANNbullet}{\cfadd{Definition:ANNcomposition}\bullet}

\newcommand{\outDimANN}{\cfadd{Def:ANN2}\mathcal{O}}

\newcommand{\dims}{\mathcal{D}}
\newcommand{\hiddenDimId}{\mathfrak{i}}

\newcommand{\ANNf}{\mathscr{f}}
\newcommand{\ANNg}{\mathscr{g}}
\newcommand{\ANNh}{\mathscr{h}}
\newcommand{\ANNi}{\mathscr{i}}
\newcommand{\ANNu}{\mathscr{u}}

\makeatletter
\newcommand{\vast}{\bBigg@{3.5}}
\newcommand{\Vast}{\bBigg@{4}}
\makeatother

\makeatletter
\DeclareFontEncoding{LS1}{}{}
\DeclareFontSubstitution{LS1}{stix}{m}{n}
\DeclareMathAlphabet{\mathscr}{LS1}{stixscr}{m}{n}
\makeatother

\usepackage{stackengine,xcolor}
\input pdf-trans
\newbox\qbox
\def\usecolor#1{\csname\string\color@#1\endcsname\space}
\newcommand\bordercolor[1]{\colsplit{1}{#1}}
\newcommand\fillcolor[1]{\colsplit{0}{#1}}
\newcommand\colsplit[2]{\colorlet{tmpcolor}{#2}\edef\tmp{\usecolor{tmpcolor}}%
	\def\tmpB{}\expandafter\colsplithelp\tmp\relax%
	\ifnum0=#1\relax\edef\fillcol{\tmpB}\else\edef\bordercol{\tmpC}\fi}
\def\colsplithelp#1#2 #3\relax{%
	\edef\tmpB{\tmpB#1#2 }%
	\ifnum `#1>`9\relax\def\tmpC{#3}\else\colsplithelp#3\relax\fi
}
\newcommand\outline[1]{\leavevmode%
	\def\maltext{#1}%
	\setbox\qbox=\hbox{\maltext}%
	\boxgs{Q q 2 Tr \thickness\space w \fillcol\space \bordercol\space}{}%
	\copy\qbox%
}
\newcommand\mathcalbb[2][1]{%
	\stackengine{0pt}{\outline{$\mathcal{#2}$}}{\kern.3pt\outline{$\mathcal{#2}$}}{O}{l}{F}{F}{L}}
\bordercolor{black}
\fillcolor{white}
\def\thickness{.1} %

\usepackage{xparse}
\usepackage{etoolbox}

\ExplSyntaxOn

\NewDocumentCommand{\enum}{ O{;} m o }
 {
  \my_enum:nnn { #1 } { #2 } { #3 }
 }

\seq_new:N \l__my_enum_seq
\tl_new:N \l__my_enum_item_tl
\prop_const_from_keyval:Nn \l__verbs
{
show = shows ,
imply = implies , 
demonstrate = demonstrates ,
prove = proves ,
establish = establishes ,
ensure = ensures ,
assure = assures
}
\int_new:N \l__number_of_args

\cs_new_protected:Nn \my_enum:nnn
 {
  \seq_set_split:Nnn \l__my_enum_seq { #1 } { #2 }
  \seq_remove_all:Nn \l__my_enum_seq {}
  \int_set_eq:NN \l__number_of_args { \seq_count:N \l__my_enum_seq }
  \seq_use:Nnnn \l__my_enum_seq { ~and~ } { ,~ } { ,~and~ }
  \IfNoValueTF{#3}{}{
    \space
    \int_compare:nNnTF{ \l__number_of_args } < {2}{ \prop_item:Nn \l__verbs {#3} }{ #3 }
  }
 }

\seq_new:N \g_cflist_loaded
\seq_new:N \g_cflist_pending

\NewDocumentCommand{\cfadd}{ m }
{
  \seq_if_in:NnF \g_cflist_loaded { #1 } {
    \seq_if_in:NnF \g_cflist_pending { #1 } {
      \seq_gput_right:Nn \g_cflist_pending { #1 }
    }
  }
}

\NewDocumentCommand{\cfload}{ o }
{
  \seq_if_empty:NTF \g_cflist_pending {\unskip} {
    (cf.\ \cref{\seq_use:Nn \g_cflist_pending {,}})\IfValueTF{#1}{#1~}{\unskip}
    \seq_gconcat:NNN \g_cflist_loaded \g_cflist_loaded \g_cflist_pending
    \seq_gclear:N \g_cflist_pending
  }
}

\NewDocumentCommand{\cfclear} {} {
  \seq_gclear:N \g_cflist_loaded
  \seq_gclear:N \g_cflist_pending
}

\NewDocumentCommand{\cfout}{ o }
{
  \seq_if_empty:NTF \g_cflist_pending {\unskip} {
    (cf.\ \cref{\seq_use:Nn \g_cflist_pending {,}})\IfValueTF{#1}{#1~}{\unskip}
    \seq_gclear:N \g_cflist_pending
  }
}

\NewDocumentCommand{\ifnocf} { m } {
  \seq_if_empty:NT \g_cflist_pending { #1 }
}

\ExplSyntaxOff

\begin{document}

\title{High-dimensional approximation spaces \\of  artificial neural networks and \\applications to partial differential equations}

\author{
	Pierfrancesco Beneventano$^{1,2}$,
	Patrick Cheridito$^3$, \\
	Arnulf Jentzen$^{4,5,6}$, and
	Philippe von Wurstemberger$^{7,8}$
	\bigskip
	\\
	\small{$^1$ Department of Mathematics, ETH Zurich,}\\ 	
	\small{Switzerland; e-mail: \texttt{pierbene96}\textcircled{\texttt{a}}\texttt{gmail.com}}
	\smallskip\\
	\small{$^2$ Department of Operations Research and Financial Engineering, }\\ 	
	\small{Princeton University, United States; e-mail: \texttt{pierb}\textcircled{\texttt{a}}\texttt{princeton.edu}}
	\smallskip\\
	\small{$^3$ Department of Mathematics, ETH Zurich,}\\ 	
	\small{Switzerland; e-mail: \texttt{patrick.cheridito}\textcircled{\texttt{a}}\texttt{math.ethz.ch}}
	\smallskip\\
  \small{$^4$ School of Data Science and Shenzhen Research Institute of Big Data, The Chinese University}\\ 	
  \small{of Hong Kong, Shenzhen (CUHK-Shenzhen), China; e-mail: \texttt{ajentzen}\textcircled{\texttt{a}}\texttt{cuhk.edu.cn}}
  \smallskip\\
  \small{$^5$ Applied Mathematics: Institute for Analysis and Numerics,}\\
  \small{University of M\"unster, Germany; e-mail: \texttt{ajentzen}\textcircled{\texttt{a}}\texttt{uni-muenster.de}}
  \smallskip\\
  \small{$^6$ Seminar for Applied Mathematics, Department of}\\
  \small{Mathematics, ETH Zurich, Switzerland}
  \smallskip\\
  \small{$^7$ School of Data Science, The Chinese University of Hong Kong,}\\
  \small{Shenzhen (CUHK-Shenzhen), China; e-mail: \texttt{philippevw}\textcircled{\texttt{a}}\texttt{cuhk.edu.cn}}
  \smallskip\\
  \small{$^8$ Department of Mathematics, ETH Zurich, Switzerland;}\\
  \small{e-mail: \texttt{philippe.vonwurstemberger}\textcircled{\texttt{a}}\texttt{math.ethz.ch}}
}

\maketitle

\begin{abstract}

In this paper we develop a new machinery to study the capacity of artificial neural networks (ANNs) to approximate high-dimensional functions without suffering from the curse of dimensionality.
Specifically, we introduce a concept which we refer to as \emph{approximation spaces of artificial neural networks} and we present several tools to handle those spaces.
Roughly speaking, approximation spaces consist of sequences of functions which can, in a suitable way, be approximated by ANNs without curse of dimensionality in the sense that the number of required ANN parameters to approximate a function of the sequence with an accuracy $\varepsilon > 0$ grows at most polynomially both in the reciprocal $1/\varepsilon$ of the required accuracy and in the dimension $d \in \N = \{1, 2, 3, \ldots \}$ of the function. 
We show 
that these approximation spaces are closed under various operations including linear combinations, formations of limits, and infinite compositions.
To illustrate the utility of the machinery proposed in this paper, we employ the developed theory to prove that ANNs have the capacity to overcome the curse of dimensionality in the numerical approximation of certain first order transport partial differential equations (PDEs). 
We even prove that approximation spaces are closed under flows of first order transport PDEs.

\end{abstract}

\tableofcontents

\section{Introduction}
\label{sect:intro}

\newcommand{\ASIntroFirst}[1]{\mathfrak{S}}
\newcommand{\ASIntro}[1]{\cfadd{Def:AS_intro}\mathfrak{S}_{#1}}
\newcommand{\realisationRelu}{\cfadd{Def:ANN_intro}\mathscr{R}}

In the last decade, the field of deep learning has achieved astonishing results by training artificial neural networks (ANNs) to perform various computational tasks in a wide range of fields including 
	image and language recognition (cf., e.g., \cite{Krizhevsky2017,Young2018,Graves2013}), 
	game intelligence (cf., e.g., \cite{guo2014deep,Silver2016}), and 
	the numerical approximation of solutions of partial differential equations (PDEs) (cf., e.g., \cite{han2018solving,Weinan2017,sirignano2018dgm}).
Accordingly, there is currently a strong interest in the scientific community to understand the success of deep learning.
Theoretical deep learning papers usually focus on different aspects of deep learning algorithms such as, for example, 
	optimization methods and training algorithms (cf., e.g., \cite{Cheridito2020a,Fehrman2019,JentzenVW18,Jentzen2020,Bottou2010,Li2015}), 
	generalization errors of ANNs (cf., e.g., \cite{Berner2020,Beck2019published,Jentzen2023c,Weinan2020,Jakubovitz2018,Advani2017}), or the  
	capacity of ANNs to approximate various kinds of functions (cf., e.g., \cite{Gonon19Uniform,Grohs2023,JentzenSalimovaWelti2021,GrohsHerrmann2020arxiv,Beck2019published,Jentzen2023c,GononSchwab20,Boelcskei2019,Petersen2017,Kutyniok2019Atheoretical,Reisinger2019Rectified,Gribonval2019,Benth2023}).

In this paper we study the capacity of artificial neural networks to approximate high-dimensional functions without suffering from the curse of dimensionality.
In the context of numerical approximations of solutions of PDEs, there have been several recent results establishing that ANNs have the capacity to overcome the curse of dimensionality when approximating solutions of various high-dimensional, possibly nonlinear, PDEs (cf., e.g., 
\cite{Beck2024,GrohsHerrmann2020arxiv,Grohs2023,JentzenSalimovaWelti2021,Hutzenthaler2019Aproof,MR4534487,Gonon19Uniform,Reisinger2019Rectified,GononSchwab20}).
The majority of those recent results show, under suitable assumptions on the initial condition, the dynamics, and the nonlinearity of some parabolic PDE, 
that the terminal value of the PDE can be approximated by neural networks with a number of parameters growing at most polynomially both in the dimension of the PDE and in the reciprocal of the required accuracy.

The aim of this work is to provide new tools to state and prove theorems on the capacity of ANNs to overcome the curse of dimensionality and thereby gain more insight into the class of functions which can be approximated by ANNs without suffering from the curse of dimensionality.
The central idea is to introduce what we refer to as \emph{approximation spaces of artificial neural networks} and to develop a theory to handle those spaces in an efficient and elegant way.
Loosely speaking, approximation spaces of ANNs consist of sequences of functions which can be approximated by ANNs without curse of dimensionality (cf.\
\cref{Def:AS_intro,Def:approximation_spaces_norates,Def:approximation_spaces_GH}).
We demonstrate that approximation spaces are closed
under 
	linear combinations (cf.\ \cref{subsect:linear_comb}), 
	formations of limits (cf.\   \cref{subsect:limits}), 
	infinite compositions  (cf.\ \cref{subsect:compositions}), and
	operations such as performing Euler steps (cf.\ \cref{subsect:Euler}).
To illustrate the utility of the machinery proposed in this paper we employ these properties to prove that approximation spaces are closed
under 
	flows of first order transport PDEs (cf.\ \cref{sect:appl_to_PDEs}  and \cref{intro_thm}).
We thereby show that ANNs have the capacity to approximate solutions of first order transport PDEs without suffering from the curse of dimensionality.

To make this more concrete, we introduce in \cref{Def:AS_intro} a simplified version of our approximation spaces  
and present in \cref{intro_thm} a result for first order transport PDEs based on those simplified approximation spaces.
For this we first recall the mathematical description of ANNs used in this paper (cf., e.g., \cite{Petersen2017,MR4534487,Jentzen2023}).

\newcommand{\Rect}{ {\cfadd{Def:ANN_intro}\mathfrak{R}}}
\cfclear
\begin{definition}[ANNs and associated ReLU realizations]
\label{Def:ANN_intro}
We denote by
$\ANNs$ the set given by 
$
	\ANNs
	= 
	\mycup_{L \in \N}
	\mycup_{ l_0,l_1,\ldots, l_L \in \N }
	\left(
	\times_{k = 1}^L (\R^{l_k \times l_{k-1}} \times \R^{l_k})
	\right)
$,
we denote by 
 $\Rect \colon \mycup_{d \in \N}  \R^d \to \mycup_{d \in \N} \R^d$ the function which satisfies for all
	$d \in \N$, 
	$x = (x_1, x_2, \ldots, x_d) \in \R^d$ 
that
\begin{equation}
\label{T_B_D}
\begin{split} 
  \Rect(x) = (\max\{x_1, 0\}, \max\{x_2, 0\}, \ldots, \max\{x_d, 0\}),
\end{split}
\end{equation}
and we denote by
	$\paramANN \colon \ANNs \to \N$ and
	$\realisationRelu \colon \ANNs \to \mycup_{k,l\in\N}C(\R^k,\R^l)$ 
the functions which satisfy for all 
	$ L\in\N$, 
	$l_0,l_1,\ldots, l_L \in \N$, 
	$
	\ANNf 
	=
	((W_1, B_1),(W_2, B_2),\allowbreak \ldots, (W_L,\allowbreak B_L))
	\in  \allowbreak
	( \times_{k = 1}^L\allowbreak(\R^{l_k \times l_{k-1}} \times \R^{l_k}))
	$,
	$x_0 \in \R^{l_0}, x_1 \in \R^{l_1}, \ldots, x_{L-1} \in \R^{l_{L-1}}$ 
with 
	$\forall \, k \in \N \cap (0,L) \colon x_k =\Rect(W_k x_{k-1} + B_k)$  
that 
$	\paramANN(\ANNf) = \sum_{k = 1}^L l_k(l_{k-1} + 1)$,
$\realisationRelu(\ANNf) \in C(\R^{l_0},\R^{l_L})$, and
\begin{equation}
\label{T_B_D}
\begin{split} 
  ( \realisationRelu(\ANNf) ) (x_0) = W_L x_{L-1} + B_L.
\end{split}
\end{equation}
\end{definition}
\noindent
The set $\ANNs$ above corresponds to the set of artificial neural networks,
the function $\Rect$ corresponds to the rectified linear unit (ReLU) activation function,
and for every 
	$\ANNf \in \ANNs$
the number $\paramANN(\ANNf) \in \N$ corresponds to the number of parameters of the ANN $\ANNf$
and the function $\realisationRelu(\ANNf)$ corresponds to the realization of the ANN $\ANNf$ with the ReLU as activation function.
We now proceed to the definition of the simplified approximation spaces considered in this introduction (see \cref{Def:approximation_spaces_norates,Def:approximation_spaces_GH} for the general definitions).

\cfclear
\begin{definition}[Simplified approximation spaces]
\label{Def:AS_intro}
Let $\kappa, \delta \in [0,\infty)$.
Then we denote by $\ASIntro{\kappa, \delta}$ the set 
given\footnote{
Recall that for all $d \in \N$, $x = (x_1, x_2, \ldots, x_d) \in \R^d$ we have that
$
	\left\| x \right\|^2 = \sum_{i = 1}^d |x_i|^2
$ (cf.\ \cref{euclidean_norm_def}).} 
by
\begin{equation}
\label{ASIntroScnd}
\begin{split}	
	\ASIntro{\kappa, \delta}
=
	\left\{ 
		(h_d)_{d \in \N} \subseteq  \mycup_{k,l \in \N} C(\R^k, \R^l) 
	\colon \left[
		\begin{array}{c}
    			\exists \, K \in \R \colon 
    			\forall \, d \in \N, \varepsilon \in (0,1] \colon \\
    			\exists \, \ANNh \in \ANNs, \mathbf{d} \in \N \colon
    			\forall \, x \in \R^{d} \colon \\
   			 \left(\begin{array}{l}
    				\{h_d, \realisationRelu(\ANNh)\} \subseteq C(\R^{d}, \R^{\mathbf{d}}), \\
    				\param(\ANNh) 
    				\leq K \varepsilon^{-K} d^K,\\
    				\norm{ h_d(x) - ( \realisationRelu(\ANNh) ) (x) }  \leq \varepsilon (1 + \norm{x}^\kappa),\\
				\delta  \Norm{( \realisationRelu(\ANNh) ) (x) } \leq K (\varepsilon^{-\delta} d^K + \norm{x})
    			\end{array} \right)
    		\end{array}
    	\right] 
	\right\}
\end{split}
\end{equation}
(cf.\ \cref{Def:ANN_intro}).
\end{definition}
\noindent
For every $\kappa, \delta \in [0,\infty)$ the set $\ASIntro{\kappa, \delta}$ consists of sequences of functions indexed over dimensions $d \in \N$ in which each function of the sequence can be approximated up to any required accuracy $\varepsilon \in  (0,1]$  by an ANN
with a number of parameters which grows at most polynomially both in the reciprocal $\nicefrac{1}{\varepsilon}$ of the required accuracy  and the dimension $d$.
Moreover, for every $\kappa \in [0,\infty)$, $\delta  \in (0,\infty)$ the realizations of the approximating ANNs for sequences of functions in $\ASIntro{\kappa,\delta}$ are only allowed to grow at most linearly with an intercept which grows at most polynomially in the dimension $d \in \N$ and at most with rate $\delta$ in the required accuracy $\varepsilon \in (0,1]$.
Note that the parameter $\kappa \in [0,\infty)$ is used to measure the accuracy of the approximating ANNs in the spaces $(\ASIntro{\kappa, \delta})_{\delta \in [0,\infty)}$.
The sets $\ASIntro{\kappa, 0}$, $\kappa \in [0,\infty)$, are simplified versions of the general approximation spaces defined in \cref{Def:approximation_spaces_norates} and
the sets $\ASIntro{\kappa, \delta}$, $\kappa \in [0,\infty)$, $\delta  \in (0,\infty)$, are simplified versions of the general approximation spaces defined in \cref{Def:approximation_spaces_GH}.
Employing the approximation spaces introduced in \cref{Def:AS_intro} above we now state in \cref{intro_thm} below a consequence of our main result on first order transport PDEs, \cref{flow_statement} below.

\cfclear
\begin{theorem}
\label{intro_thm}
Let 
	$T, c \in (0,\infty)$, $\kappa \in [1,\infty)$, $\delta \in (0,\nicefrac{1}{\kappa})$,
	$(f_d)_{d \in \N} \in \ASIntro{\kappa, \delta}$,
let $u_d \in C^1([0, T] \times \R^d, \R)$, $d \in \N$, satisfy $(u_d(0, \cdot))_{d \in \N} \in \ASIntro{\kappa, 0}$,
and assume for all
	$d \in \N$,
	$t \in [0, T]$,
	$x, y \in \R^d$
that 
$f_d \in C^1(\R^d, \R^d)$,
$\norm{f_d(0)} \leq cd^c$,	
$
    \norm{ f_{d}(x) - f_{d}(y) } + |u_d(0, x) - u_d(0, y)|
    \leq c\Norm{ x - y }
$,
and
\begin{equation}
\label{intro_thm:ass1}
\begin{split} 
  \tfrac{\partial  u_d}{\partial t}  (t,x)  
=
  \tfrac{\partial u_d}{\partial x} (t,x)\, f_d(x)
\end{split}
\end{equation}
(cf.\ \cref{Def:AS_intro}).
Then it holds for all $t \in [0, T]$ that 
\begin{equation}
\label{T_B_D}
\begin{split} 
  (u_d(t, \cdot))_{d \in \N} \in \ASIntro{\kappa, 0}
\end{split}
\end{equation}
\end{theorem}
\noindent
\cref{intro_thm} is a direct consequence of \cref{relu_flow_statement_max} in Subsection \ref{subsect:PDE_flows_relu}. 
\cref{relu_flow_statement_max}, in turn, follows from \cref{flow_statement} in \cref{subsect:PDE_flows_in_appr_space}.
Note that the assumptions 
	on the drift functions $(f_d)_{d \in \N}$ of the PDEs in \eqref{intro_thm:ass1} and 
	on the initial values $(u_d(0, \cdot))_{d \in \N}$ of the PDEs in \eqref{intro_thm:ass1}
as well as the conclusion of \cref{intro_thm} are conveniently formulated in terms of the approximation spaces introduced in \cref{Def:AS_intro}.
Loosely speaking, \cref{intro_thm} states that if the sequence of initial conditions of the first order transport PDEs in \eqref{intro_thm:ass1} are contained in the approximation space $\ASIntro{\kappa, 0}$, then this property is preserved along the flows of these PDEs.
To the best of our knowledge, \cref{intro_thm} is the only theorem about approximation capacities of ANNs for PDEs which measures the approximation errors of ANNs based on a supremal condition on the entire euclidean space. 
Most papers in the scientific literature consider approximation errors in the $L^p$-sense (cf., e.g., \cite{Grohs2023,Hutzenthaler2019Aproof,Reisinger2019Rectified}) and some in the supremum sense but on a compact set (cf., e.g., \cite{Beck2024,GrohsHerrmann2020arxiv,Gonon19Uniform,GononSchwab20}).

The remainder of this article is organized as follows.
In \cref{sect:ANN_calculus} we recall the definition of ANNs (cf.\ \cref{subsect:ANN_def}) and present elementary properties of operations with ANNs such as compositions of ANNs (cf.\ \cref{subsect:ANN_comp}) and sums of ANNs (cf.\ \cref{subsect:ANN_sum,subsect:ANN_Euler}).
In \cref{sect:AS} we introduce the notion of approximation spaces of ANNs 
(cf.\ \cref{subsect:functionfamilies,subsect:dimmappings,subsect:AS}), 
the central concept of this paper, 
and develop a theory for those spaces (cf.\ \cref{subsect:linear_comb,subsect:limits,subsect:compositions,subsect:Euler}).
In \cref{sect:appl_to_PDEs} we consider the flow of first order transport PDEs (cf.\ \cref{subsect:PDE_flow}) and 
show how the  theory developed in \cref{sect:AS}
combined with 
the Euler scheme (cf.\ \cref{subsect:euler_scheme}) can be employed to
prove results on the approximation capacity of ANNs in the case of first order transport PDEs (cf.\ \cref{subsect:PDE_flows_in_appr_space,subsect:PDE_flows_relu}).

\section{Artificial neural network (ANN) calculus}
\label{sect:ANN_calculus}

In this section, we introduce and discuss some concepts and operations related to the set of ANNs $\ANNs$ presented in \cref{Def:ANN_intro} such as
	realizations of ANNs for a general activation function (cf.\ \cref{subsect:ANN_def}),
 	compositions of ANNs (cf.\ \cref{subsect:ANN_comp}), 
 	sums of ANNs (cf.\ \cref{subsect:ANN_sum}), and 
 	the existence of ANNs emulating Euler steps (cf.\ \cref{subsect:ANN_Euler}).
This section is an extension of the calculus for ANNs developed in Grohs et al.\ \cite{MR4534487}.
In some results we consider activation functions which allow the identity function to be efficiently represented by a neural network with one hidden layer (cf.\ \cref{Lemma:SumsOfANNS} and \cref{Lemma:EulerWithANNS}). The most common activation functions for which this is the case are the ReLU activation function and leaky ReLU activation functions (cf.\ \cref{identity_representation} in \cref{sect:appl_to_PDEs}).

\subsection[ANNs and their realizations]{Artificial neural networks and their realizations}
\label{subsect:ANN_def}

\begin{definition}[Architecture mappings of ANNs]
	\label{Def:ANN2}
	We denote by 
	$\lengthANN \colon \ANNs \to \N$, 
	$\inDimANN \colon \ANNs \to \N$, 
	$\outDimANN \colon \ANNs \to \N
	$, $\hidLengthAnn\colon \ANNs \to \N_0$, and
	$\dims\colon\ANNs\to  \mycup_{L=2}^\infty \N^{L}$
	 the functions which satisfy
	for all $ L\in\N$, $l_0,l_1,\ldots, l_L \in \N$, 
	$
	\ANNf 
	\in  \allowbreak
	( \times_{k = 1}^L\allowbreak(\R^{l_k \times l_{k-1}} \times \R^{l_k}))$
	that
	$\lengthANN(\ANNf)=L$, 
	$\inDimANN(\ANNf)=l_0$,
	$\outDimANN(\ANNf)=l_L$,
	$\hidLengthAnn(\ANNf)=L-1$, and
	$\dims(\ANNf)= (l_0,l_1,\ldots, l_L)$
(cf.\ \cref{Def:ANN_intro}).
\end{definition}

\newcommand{\multdim}{\cfadd{Def:multidim_version}\mathfrak M}
\begin{definition}[Multidimensional versions]
	\label{Def:multidim_version}
	Let $a \colon \R \to \R$ be a function.
	Then we denote by $\multdim_{a} \colon \mycup_{d \in \N}  \R^d \to \mycup_{d \in \N} \R^d$ the function which satisfies for all $d\in\N$, $ x = ( x_1, \dots, x_{d} ) \in \R^{d} $ that
$	\multdim_{a}( x ) 
	=
	\left(
	a(x_1)
	,
	\ldots
	,
	a(x_d)
	\right)
$.
\end{definition}

\cfclear
\begin{definition}[Realizations associated to ANNs]
	\label{Def:ANNrealization}
	Let $a\in C(\R,\R)$.
	Then we denote by 
	$
	\functionANN \colon \ANNs \to \mycup_{k,l\in\N} C(\R^k,\R^l)
	$
	the function which satisfies
	for all  $ L\in\N$, $l_0,l_1,\ldots, l_L \in \N$, 
	$
	\ANNf 
	=
	((W_1, B_1),(W_2, B_2),\allowbreak \ldots, (W_L,\allowbreak B_L))
	\in  \allowbreak
	( \times_{k = 1}^L\allowbreak(\R^{l_k \times l_{k-1}} \times \R^{l_k}))
	$,
	$x_0 \in \R^{l_0}, x_1 \in \R^{l_1}, \ldots, x_{L-1} \in \R^{l_{L-1}}$ 
	with $\forall \, k \in \N \cap (0,L) \colon x_k =\multdim_{a}(W_k x_{k-1} + B_k)$  
	that
$
	\functionANN(\ANNf) \in C(\R^{l_0},\R^{l_L})$ and $
	( \functionANN(\ANNf) ) (x_0) = W_L x_{L-1} + B_L
$
	(cf.\ \cref{Def:multidim_version,Def:ANN_intro}).
\end{definition}

\subsection{Compositions of ANNs}
\label{subsect:ANN_comp}
\cfclear
\begin{definition}[Standard compositions of ANNs]
	\label{Definition:ANNcomposition}
	We denote by $\compANN{(\cdot)}{(\cdot)}\colon\allowbreak \{(\ANNf_1,\ANNf_2)\allowbreak\in\ANNs\times \ANNs\colon \inDimANN(\ANNf_1)=\outDimANN(\ANNf_2)\}\allowbreak\to\ANNs$ the function which satisfies for all 
	$ L,\mathfrak{L}\in\N$, $l_0,l_1,\ldots, l_L, \mathfrak{l}_0,\mathfrak{l}_1,\ldots, \mathfrak{l}_\mathfrak{L} \in \N$, 
	$
	\ANNf_1
	=
	((W_1, B_1),(W_2, B_2),\allowbreak \ldots, (W_L,\allowbreak B_L))
	\in  \allowbreak
	( \times_{k = 1}^L\allowbreak(\R^{l_k \times l_{k-1}} \times \R^{l_k}))
	$,
	$
	\ANNf_2
	=
	((\mathfrak{W}_1, \mathfrak{B}_1),\allowbreak(\mathfrak{W}_2, \mathfrak{B}_2),\allowbreak \ldots, (\mathfrak{W}_\mathfrak{L},\allowbreak \mathfrak{B}_\mathfrak{L}))
	\in  \allowbreak
	( \times_{k = 1}^\mathfrak{L}\allowbreak(\R^{\mathfrak{l}_k \times \mathfrak{l}_{k-1}} \times \R^{\mathfrak{l}_k}))
	$ 
	with $l_0=\inDimANN(\ANNf_1)=\outDimANN(\ANNf_2)=\mathfrak{l}_{\mathfrak{L}}$
	that
	\begin{equation}\label{ANNoperations:Composition}
	\begin{split}
	&\compANN{\ANNf_1}{\ANNf_2}=
	\begin{cases} 
			\begin{array}{r}
			\big((\mathfrak{W}_1, \mathfrak{B}_1),(\mathfrak{W}_2, \mathfrak{B}_2),\ldots, (\mathfrak{W}_{\mathfrak{L}-1},\allowbreak \mathfrak{B}_{\mathfrak{L}-1}),
			(W_1 \mathfrak{W}_{\mathfrak{L}}, W_1 \mathfrak{B}_{\mathfrak{L}}+B_{1}),\\ (W_2, B_2), (W_3, B_3),\ldots,(W_{L},\allowbreak B_{L})\big)
			\end{array}
	&: L>1<\mathfrak{L} \\[3ex]
	\big( (W_1 \mathfrak{W}_{1}, W_1 \mathfrak{B}_1+B_{1}), (W_2, B_2), (W_3, B_3),\ldots,(W_{L},\allowbreak B_{L}) \big)
	&: L>1=\mathfrak{L}\\[1ex]
	\big((\mathfrak{W}_1, \mathfrak{B}_1),(\mathfrak{W}_2, \mathfrak{B}_2),\allowbreak \ldots, (\mathfrak{W}_{\mathfrak{L}-1},\allowbreak \mathfrak{B}_{\mathfrak{L}-1}),(W_1 \mathfrak{W}_{\mathfrak{L}}, W_1 \mathfrak{B}_{\mathfrak{L}}+B {1}) \big)
	&: L=1<\mathfrak{L}  \\[1ex]
	(W_1 \mathfrak{W}_{1}, W_1 \mathfrak{B}_1+B_{1}) 
	&: L=1=\mathfrak{L} 
	\end{cases}
	\end{split}
	\end{equation}
	(cf.\ \cref{Def:ANN_intro,Def:ANN2}).
\end{definition}

\begin{samepage}

\cfclear
\begin{prop}\label{Lemma:PropertiesOfCompositions_n2}
Let 
	$n \in \N$,
	$\ANNf_1, \ANNf_2, \ldots, \ANNf_n \in \ANNs$
satisfy for all $k\in\N \cap (0,n)$ that
	$\inDimANN(\ANNf_k)=\outDimANN(\ANNf_{k+1})$
and let 
	$(l_{k,j})_{(k,j) \in \{ (\mathbf{k}, \mathbf{j}) \in  \{1, 2, \ldots, n\} \times \N_0 \colon \mathbf{j} \leq \lengthANN({\ANNf_{\mathbf{k}}})\}} \subseteq \N$
satisfy for all 
	$k\in\{1, 2, \ldots, n\}$
that
	$\dims(\ANNf_k)=(l_{k,0},l_{k,1},\dots, l_{k,\lengthANN(\ANNf_k)})$
\cfload.
	Then
\begin{enumerate}[(i)]

\item \label{PropertiesOfCompositions_n:Input} 
 it holds that
$
\inDimANN(\ANNf_1 \compANNbullet \ANNf_2\compANNbullet \ldots \compANNbullet \ANNf_n)=\inDimANN(\ANNf_n),
$

\item \label{PropertiesOfCompositions_n:Output} 
 it holds that
$
\outDimANN(\ANNf_1 \compANNbullet \ANNf_2\compANNbullet \ldots \compANNbullet \ANNf_n)=\outDimANN(\ANNf_1),
$
\item \label{PropertiesOfCompositions_n:Realization} 
 it holds for all  $\activation\in C(\R,\R)$ that
\begin{equation}\label{PropertiesOfCompositions:RealizationEquation_n}
\functionANN({\ANNf_1}\compANNbullet{\ANNf_2}\compANNbullet \ldots \compANNbullet \ANNf_n)=\functionANN(\ANNf_1)\circ \functionANN(\ANNf_2)\circ \ldots \circ \functionANN(\ANNf_n),
\end{equation}

\item \label{PropertiesOfCompositions_n:Dims} it holds that
\begin{multline}
     \dims({\ANNf_1}\compANNbullet{\ANNf_2}\compANNbullet \ldots \compANNbullet \ANNf_n)
    \\  = (l_{n,0},l_{n,1},\dots, l_{n,\lengthANN(\ANNf_n)-1},l_{n-1,1},\dots,l_{n-1,\lengthANN(\ANNf_{n-1})-1},l_{n-2,1},\dots,l_{n-2,\lengthANN(\ANNf_{n-2})-1},\ldots,
    \\  \quad l_{1,1}, \ldots, l_{1,\lengthANN(\ANNf_1)-1}, l_{1,\lengthANN(\ANNf_1)} )
\end{multline}
and
\item \label{PropertiesOfCompositions_n:Params} it holds that
\begin{equation}
\begin{split}
\paramANN({\ANNf_1}\compANNbullet{\ANNf_2}\compANNbullet \ldots \compANNbullet \ANNf_n)
\leq
\left[ \sum_{k=1}^n\paramANN(\ANNf_k) \right]
+ \left[ \sum_{k = 1}^{n-1}l_{k,1}( l_{k+1,\lengthANN(\ANNf_{k+1})-1}+1)\right]
\end{split}
\end{equation}
\end{enumerate}
\cfout.
\end{prop}

\begin{proof}%
Observe that, e.g., Grohs et al.\ \cite[Proposition 2.6 and Lemma 2.8]{MR4534487}
and induction establish items~\eqref{PropertiesOfCompositions_n:Input}--\eqref{PropertiesOfCompositions_n:Params}.
The proof of \cref{Lemma:PropertiesOfCompositions_n2} is thus complete.
\end{proof}

\end{samepage}

\cfclear
\begin{cor}
\label{Lemma:PropertiesOfCompositions_n3}
    Let $n \in \{2, 3, \ldots \}$,
    $\ANNf_1, \ANNf_2, \ldots, \ANNf_n \in \ANNs$
	satisfy for all $k\in \N \cap (0,n)$ that
	$\inDimANN(\ANNf_k)=\outDimANN(\ANNf_{k+1})$
	\cfload.
	Then 
\begin{equation}
\label{PropertiesOfCompositions_n:Params2}
\paramANN({\ANNf_1}\compANNbullet{\ANNf_2}\compANNbullet \ldots \compANNbullet \ANNf_n)
\leq 2 \left( \sum_{k=1}^{n-1}\paramANN(\ANNf_k)\paramANN(\ANNf_{k+1}) \right)
\end{equation}  
\cfout.
\end{cor}

\begin{proof}
\cref{Lemma:PropertiesOfCompositions_n3} is a consequence of item~\eqref{PropertiesOfCompositions_n:Params} in \cref{Lemma:PropertiesOfCompositions_n2}. See, e.g., \cite[Corollary 2.8]{Beneventano2020v1} for a detailed proof.
\end{proof}

\subsection{Sums of ANNs}
\label{subsect:ANN_sum}

\cfclear
\begin{prop}[Sums of ANNs]\label{Lemma:SumsOfANNS}
Let 
$a\in C(\R,\R)$, 
$\lambda_1, \lambda_2, c \in\R$,
$\ANNf_1,\ANNf_2 \in\ANNs$, 
$(\hiddenDimId_{d})_{d \in \N} \subseteq \N$, 
$(\ANNi_{d})_{d \in \N} \subseteq \ANNs$ satisfy for all 
	$d \in \N$, $x\in\R^{d}$
 that 
$\dims(\ANNi_{d}) = (d,\hiddenDimId_{d},d)$, $\hiddenDimId_{d} \leq c d$, $(\functionANN(\ANNi_{d}))(x)=x$, 
$\inDimANN(\ANNf_1)=\inDimANN(\ANNf_2)$, and $\outDimANN(\ANNf_1) = \outDimANN(\ANNf_2)$
\cfload.
	Then 
	there exists $\ANNh\in \ANNs$
 such that
 \begin{enumerate}[(i)]
 	\item \label{SumsOfANNS:item1} 
 	it holds that $\functionANN (\ANNh)\in C(\R^{\inDimANN(\ANNf_1)},\R^{\outDimANN(\ANNf_1)})$,
 	\item \label{SumsOfANNS:item2} 
 	it holds for all
 	$x\in\R^{\inDimANN(\ANNf_1)}$ that  
	$
 	(\functionANN (\ANNh))(x)=\lambda_1(\functionANN (\ANNf_1))(x) + \lambda_2(\functionANN (\ANNf_2))(x),
 	$
 	and
 	\item \label{SumsOfANNS:item3} it holds that 
 	$
 		\paramANN(\ANNh)
 	\leq 
 	11 \max \{1, c^2\} (\max \{ \inDimANN(\ANNf_1), \outDimANN(\ANNf_1)\})^2
 	(\paramANN(\ANNf_1) + \paramANN(\ANNf_2))
 	$.
 \end{enumerate}		
\end{prop}

\begin{proof}%
\cref{Lemma:SumsOfANNS} is a simple consequence of the parallelization of ANNs presented in Cheridito et al.\ \cite[Proposition II.5]{Cheridito2019}. See, e.g., \cite[Proposition 2.9]{Beneventano2020v1} for a detailed proof.
\end{proof}

\subsection{ANN emulations for Euler steps}
\label{subsect:ANN_Euler}

\cfclear
\begin{prop}[ANNs for Euler steps]\label{Lemma:EulerWithANNS}
	Let 
	$a\in C(\R,\R)$, 
	$c \in (0, \infty)$, 
	$\delta \in \R$,
	$\ANNf \in \ANNs$, 
	$(\hiddenDimId_{d})_{d \in \N} \subseteq \N$, 
	$(\ANNi_{d})_{d \in \N} \subseteq \ANNs$  satisfy for all 
		$d \in \N$, $x\in\R^{d}$ 
	that  
	$\inDimANN(\ANNf)=\outDimANN(\ANNf)$,
	$\dims(\ANNi_{d}) = (d,\hiddenDimId_{d},d)$, $\hiddenDimId_{d} \leq c d$, and
	$(\functionANN(\ANNi_{d}))(x)=x$
	\cfload.
	Then 
	there exists $\ANNh\in \ANNs$ 
 such that
 \begin{enumerate}[(i)]
 	\item \label{EulerWithANNS:Continuity} it holds that $\functionANN (\ANNh)\in C(\R^{\inDimANN(\ANNf)},\R^{\inDimANN(\ANNf)})$,
 	\item \label{EulerWithANNS:Realization} it holds for all
 	$x\in\R^{\inDimANN(\ANNf)}$ that  
 	$
 	    (\functionANN (\ANNh))(x)= x + \delta (\functionANN (\ANNf))(x)
 	$, and
 	\item \label{EulerWithANNS:Params1} it holds that 
	$
 	\paramANN(\ANNh)
 	 \leq 
 	 11 \max \{1, c^2\} (\inDimANN(\ANNf))^2 \paramANN(\ANNi_{\inDimANN(\ANNf)}) \paramANN(\ANNf)
 	\leq
 	44 \max \{1, c^3\} (\inDimANN(\ANNf))^4 \paramANN(\ANNf)
 	$.
 \end{enumerate}	
\end{prop}

\begin{proof}%
\cref{Lemma:EulerWithANNS} is a consequence of \cref{Lemma:SumsOfANNS}. See, e.g., \cite[Proposition 2.10]{Beneventano2020v1} for a detailed proof.
\end{proof}

\section{Approximation spaces of artificial neural networks}
\label{sect:AS}

It is the subject of this section to introduce the concept of \emph{approximation spaces of artificial neural networks} (see \cref{Def:approximation_spaces_norates,Def:approximation_spaces_GH}) and to develop a machinery to handle those spaces. 
Unlike the simplified approximation spaces defined in the introduction of this paper (see \cref{Def:AS_intro} above) which consisted of sequences of functions, elements of the general approximation spaces 
will be what we call \emph{function families} and define in \cref{subsect:functionfamilies}.
Function families are collections of functions indexed over some index set $I$.
In order to have a notion of dimensionality for function families %
we present the concept of \emph{dimension mappings} in \cref{subsect:dimmappings}.  
Dimension mappings are functions $\mathfrak{d} \colon I \to \N^N$ for some $N \in \N$, which assign to each index $i \in I$ a suitable vector $\mathfrak{d}(i) \in \N^N$ encoding the dimensionality of the $i$-th function of function families on $I$.
With function families and dimension mappings at hand we define approximation spaces and discuss some of their elementary properties in \cref{subsect:AS}.

The remaining \cref{subsect:compositions,subsect:linear_comb,subsect:Euler,subsect:limits} are devoted to the development of a theory for approximation spaces.
In \cref{subsect:linear_comb} we demonstrate that linear combinations of elements of approximation spaces are again contained in approximation spaces thereby implying that approximation spaces can be given the structure of a vector space.
In \cref{subsect:limits} we show that
limits of function families in approximation spaces are again contained in approximation spaces.
In \cref{subsect:compositions} we reveal that
infinite compositions of function families in approximation spaces are again contained in approximation spaces.
Finally, in \cref{subsect:Euler} we show that
function families consisting of Euler approximations induced by function families in an approximation space, are again contained in approximation spaces.

\subsection{Function families}
\label{subsect:functionfamilies}
\subsubsection{Definition of function families}

\begin{definition}[Function families]
\label{Def:functionfamily}
We say that $f$ is a function family on $I$ if and only if
it holds
\begin{enumerate}[(i)]
\item \label{Def:functionfamily:item1}
that $I$ is a non-empty set and
\item \label{Def:functionfamily:item2}
that
$f \colon I \to \mycup_{k,l \in \N} C(\R^k, \R^l)$ is a function from $I$ to $ \mycup_{k,l \in \N} C(\R^k, \R^l)$.
\end{enumerate}
\end{definition}

\cfclear
\newcommand{\setofFF}{\cfadd{Def:set_of_functionfamilies}\mathfrak{F}}
\begin{definition}[Sets of all function families]
\label{Def:set_of_functionfamilies}
Let $I$ be a non-empty set.
Then we denote by $\setofFF_I$ the set of all \cfadd{Def:functionfamily}function families on $I$ (cf.\ \cref{Def:functionfamily}).
\end{definition}

\subsubsection{Operations on function families}

\newcommand{\compofstruct}{\cfadd{composition_of_FF_def} \circledcirc}
\cfclear
\begin{definition}[Compositions of function families]
\label{composition_of_FF_def}
Let $I$ be a non-empty set, let $f = (f_i)_{i \in I}$, $g = (g_i)_{i \in I} \in \setofFF_I$,
and assume  
for all $i \in I$ that $\domain(f_i) = \codomain(g_i)$ \cfload.
Then we denote by $ f \compofstruct g \in \setofFF_I$
the \cfadd{Def:functionfamily}function family on $I$ which satisfies 
for all $i \in I$ that 
\begin{equation}
  \left( f \compofstruct g \right)_i 
=
  f_i \circ g_i
\end{equation}
(cf.\ \cref{Def:set_of_functionfamilies}).
\end{definition}

\newcommand{\prodofstruct}{\cfadd{product_of_functionfamilies}\mathop{\boxtimes}}
\cfclear
\begin{definition}[Products of function families]
\label{product_of_functionfamilies}
Let $I$ and $J$ be non-empty sets and
for every $j \in J$ let $f^j = (f^j_i)_{i \in I} \in \setofFF_I$ \cfload.
Then we denote by 
$ \prodofstruct_{j \in J} \! f^{j} \in \setofFF_{I \times J}$ the \cfadd{Def:functionfamily}function family on $I \times J$ which satisfies 
for all $\mathbf{i} \in I$, $\mathbf{j} \in J$ that 
\begin{equation}
  {\textstyle \big( \prodofstruct_{j \in J} f^{j} \big)_{\mathbf{i}, \mathbf{j}} }
=
  f^{\mathbf{j}}_\mathbf{i}
\end{equation}
(cf.\ \cref{Def:set_of_functionfamilies}).
\end{definition}

\subsection{Dimension mappings}
\label{subsect:dimmappings}

\subsubsection{Definition of dimension mappings}

\begin{definition}[$N$-fold dimension mappings]
\label{dimension_mapping_1}
We say that $\mathfrak{d}$ is an $N$-fold dimension-mapping on $I$ if and only if
it holds 
\begin{enumerate}[(i)]
\item \label{dimension_mapping_1:item1}
that $I$ is a non-empty set,

\item \label{dimension_mapping_1:item2}
that $N \in \N$ is a natural number, and

\item \label{dimension_mapping_1:item3}
that $\mathfrak{d} \colon I \to \N^N$ is a function from $I$ to $ \N^N$.
\end{enumerate}
\end{definition}

\subsubsection{Operations on dimension mappings}

\newcommand{\proddimmap}{\cfadd{product_of_dimension_mappings}\otimes}
\cfclear
\begin{definition}[Product of dimension-mappings]
\label{product_of_dimension_mappings}
Let $\mathfrak{d} = (\mathfrak{d}_1, \mathfrak{d}_2, \ldots, \mathfrak{d}_N)$ be an $N$-fold \linebreak dimension-mapping\cfadd{dimension_mapping_1} on $I$ and 
let $\mathsf{d} = (\mathsf{d}_1, \mathsf{d}_2, \ldots, \mathsf{d}_M)$ be an $M$-fold dimension-mapping\cfadd{dimension_mapping_1} on $J$ \cfload.
Then we denote by 
$\mathfrak{d} \proddimmap \mathsf{d} \colon I \times J \to  \N^{N+M}$ 
the $(N+M)$-fold dimension mapping\cfadd{dimension_mapping_1} on $I \times J$ which satisfies 
for all $i \in I$, $j \in J$ that
\begin{equation}
  (\mathfrak{d} \proddimmap \mathsf{d})(i,j) 
=
  (\mathfrak{d}_1(i), \mathfrak{d}_2(i), \ldots, \mathfrak{d}_N(i), \mathsf{d}_1(j), \mathsf{d}_2(j), \ldots, \mathsf{d}_M(j)) \in \N^{N+M}.
\end{equation}
\end{definition}

\subsection{Approximation spaces}
\label{subsect:AS}

\subsubsection{Definition of approximation spaces}

\begin{definition}[Standard norm]
\label{euclidean_norm_def}
We denote by $\newnorm{\cdot} \colon (\mycup_{d \in \N}\R^d) \to [0,\infty)$ the 
function which satisfies for all $d \in \N$, $x = (x_1, x_2, \ldots, x_d) \in \R^d$ that
$
    \norm{x} = \big[ \sum_{i = 1}^d |x_i|^2 \big]^{\nicefrac{1}{2}}
$.
\end{definition}

\newcommand{\ASrates}[2]{\cfadd{Def:approximation_spaces_GH} \mathcal{W}_{#1 , \infty, \infty}^{#2}}

\newcommand{\ASGH}{{\cfadd{Def:approximation_spaces_GH} \mathcal{W}}}
\cfclear
\begin{definition}[Approximation spaces]
\label{Def:approximation_spaces_GH}
Let $\mathfrak{d} = (\mathfrak{d}_1, \mathfrak{d}_2, \ldots, \mathfrak{d}_N)$ be an N-fold dimension mapping\cfadd{dimension_mapping_1} on $I$, 
let 
	$r_0, r_1, \ldots, r_N \in \R$, 
	$a \in C(\R, \R)$, 
	$G = (G_{i,\varepsilon})_{(i,\varepsilon) \in I\times (0,1]} \subseteq (0, \infty]$, 
	$H = (H_{i,\varepsilon})_{(i,\varepsilon) \in I\times (0,1]} \subseteq (0, \infty]$,
and let $\mathfrak{w} \colon [0, \infty) \to [0, \infty)$ be a function
 \cfload.
Then we denote by 
$ \ASGH_{a, \mathfrak{w}, G, H }^{r_0, r_1, \ldots, r_N}(I, \mathfrak{d}) $
the set given by
\begin{equation}
\label{approximation_spaces_GH:eq1}
\begin{split}
    &\ASGH_{a, \mathfrak{w}, G, H }^{r_0, r_1, \ldots, r_N}(I, \mathfrak{d}) \cfclear
    = \left\{ f = (f_i)_{i \in I}
    \in \setofFF_I \colon \left[
    \begin{array}{c}
    	\exists \, K \in \R \colon 
    	\forall \, i \in I, \varepsilon \in (0,1] \colon \\
    	\exists \, \ANNf \in \ANNs, \mathfrak{i}, \mathfrak{o} \in \N \colon 
    	\forall \, x \in \R^{\mathfrak{i}}\colon \\
    \left(\begin{array}{l}
    	\{f_i, \functionANN(\ANNf)\} \subseteq C(\R^{\mathfrak{i}}, \R^{\mathfrak{o}}),\\
	    \param(\ANNf) 
	    \leq K \varepsilon^{-r_0}
	    \prod_{l= 1}^N |\mathfrak{d}_l(i)|^{r_l},\\ 
	     \mathfrak{w}( \norm{ x })
	    \norm{ f_i(x) - ( \functionANN(\ANNf) ) (x) }  \leq \varepsilon,\\
	    \norm{( \functionANN(\ANNf) ) (x)}
	    \leq G_{i,\varepsilon} (H_{i,\varepsilon} + \Norm{x})
    \end{array} \right)
    \end{array}
    \right] \right\}
\end{split}
\end{equation}
\cfout.
\end{definition}

\cfclear
\begin{lemma}[Approximation spaces without linear growth conditions]
\label{ASRates_deflemma}
Let $\mathfrak{d} = (\mathfrak{d}_1, \mathfrak{d}_2, \ldots, \mathfrak{d}_N)$ be an N-fold dimension mapping\cfadd{dimension_mapping_1} on $I$, let $r_0, r_1, \ldots, r_N \in \R$,  
$a \in C(\R, \R)$, and let $\mathfrak{w} \colon [0, \infty) \to [0, \infty)$ be a function \cfload.
Then 
\begin{equation}
\label{ASRates_deflemma:eq1}
\begin{split} 
	&\ASrates{a, \mathfrak{w}}{r_0, r_1, \ldots, r_N}(I, \mathfrak{d}) \\
&=
	\left\{ f = (f_i)_{i \in I}
	    \in \setofFF_I \colon \left[
	    \begin{array}{c}
	    \exists \, K \in \R \colon 
	    \forall \, i \in I, \varepsilon \in (0,1] \colon 
	    \exists \, \ANNf \in \ANNs, \mathfrak{i}, \mathfrak{o} \in \N \colon \\
	    \left(\begin{array}{l}
	    \{f_i, \functionANN(\ANNf)\} \subseteq C(\R^{\mathfrak{i}}, \R^{\mathfrak{o}}),\\
	    \param(\ANNf) 
	    \leq K \varepsilon^{-r_0}
	     \prod_{l= 1}^N |\mathfrak{d}_l(i)|^{r_l},\\
	    \sup_{x \in \R^{\mathfrak{i}}} 
	    (\mathfrak{w}( \norm{ x })
	    \norm{ f_i(x) - ( \functionANN(\ANNf) ) (x) })  \leq \varepsilon
	    \end{array} \right)
	    \end{array}
	    \right] \right\}
\end{split}
\end{equation}
\cfout.
\end{lemma}

\begin{proof}%
Note that \eqref{approximation_spaces_GH:eq1} implies \eqref{ASRates_deflemma:eq1}.
The proof of \cref{ASRates_deflemma} is thus complete.
\end{proof}

\cfclear
\newcommand{\ASnorates}{\cfadd{Def:approximation_spaces_norates} \mathcalbb{W}}
\begin{definition}[Approximation spaces without linear growth conditions nor explicit rates]
\label{Def:approximation_spaces_norates}
Let $\mathfrak{d} = (\mathfrak{d}_1, \mathfrak{d}_2, \ldots, \mathfrak{d}_N)$ be an N-fold dimension mapping\cfadd{dimension_mapping_1} on $I$, let $a \in C(\R, \R)$, and let $\mathfrak{w} \colon [0, \infty) \to [0, \infty)$ be a function \cfload.
Then we denote by 
$
  \ASnorates_{a, \mathfrak{w} }(I,\mathfrak{d})
$
the set given by
\begin{equation}
\label{approximation_spaces_norates:eq1}
  \ASnorates_{a, \mathfrak{w} }(I,\mathfrak{d})
=
  \mathop{\cup}\limits_{r_0, r_1, \ldots, r_N \in \R} 
  \ASrates{a, \mathfrak{w} }{r_0, r_1, \ldots, r_N}(I,\mathfrak{d})
\end{equation}
(cf.\ \cref{Def:approximation_spaces_GH}).
\end{definition}

\cfclear
\begin{lemma}
\label{AS_norates_alternative_def}
Let $\mathfrak{d} = (\mathfrak{d}_1, \mathfrak{d}_2, \ldots, \mathfrak{d}_N)$ be an N-fold dimension mapping\cfadd{dimension_mapping_1} on $I$, let $a \in C(\R, \R)$, and let $\mathfrak{w} \colon [0, \infty) \to [0, \infty)$ be a function \cfload.
Then
\begin{equation}
\label{AS_norates_alternative_def:concl}
\begin{split} 
	\ASnorates_{a, \mathfrak{w} }(I,\mathfrak{d})
	=
	  \left\{ f = (f_i)_{i \in I}
	  	    \in \setofFF_I \colon \left[
	  	    \begin{array}{c}
	  	    \exists \, K \in \R \colon 
	  	    \forall \, i \in I, \varepsilon \in (0,1] \colon 
	  	    \exists \, \ANNf \in \ANNs, \mathfrak{i}, \mathfrak{o} \in \N \colon \\
	  	    \left(\begin{array}{l}
	  	    \{f_i, \functionANN(\ANNf)\} \subseteq C(\R^{\mathfrak{i}}, \R^{\mathfrak{o}}),\\
	  	    \param(\ANNf) 
	  	    \leq K \varepsilon^{-K}
	  	     \prod_{l= 1}^N |\mathfrak{d}_l(i)|^{K},\\
	  	    \sup_{x \in \R^{\mathfrak{i}}} 
	  	    (\mathfrak{w}( \norm{ x })
	  	    \norm{ f_i(x) - ( \functionANN(\ANNf) ) (x) })  \leq \varepsilon
	  	    \end{array} \right)
	  	    \end{array}
	  	    \right] \right\}
\end{split}
\end{equation}
\cfout.
\end{lemma}
	
\begin{proof}%
Note that \eqref{approximation_spaces_GH:eq1} and \eqref{approximation_spaces_norates:eq1} establish \eqref{AS_norates_alternative_def:concl}.
The proof of \cref{AS_norates_alternative_def} is thus complete.
\end{proof}

\subsubsection{Elementary properties of approximation spaces}

\begin{lemma}[Characterization of approximation spaces]
\cfclear
\label{approximation_space_equivalenceGH}
Let $\mathfrak{d} = (\mathfrak{d}_1, \mathfrak{d}_2, \ldots, \mathfrak{d}_N)$ be an N-fold dimension mapping\cfadd{dimension_mapping_1} on $I$,
let 
	$r_0, r_1, \ldots, r_N \in \R$, 
	$a \in C(\R, \R)$, 
	$G = (G_{i,\varepsilon})_{(i,\varepsilon) \in I\times (0,1]} \subseteq (0, \infty]$, 
	$H = (H_{i,\varepsilon})_{(i,\varepsilon) \in I\times (0,1]} \subseteq (0, \infty]$,
	$f = (f_i)_{i \in I} \in \setofFF_I$,
and let $\mathfrak{w} \colon [0, \infty) \to [0, \infty)$,
$\inputStruc \colon I \to \N$, and $\outputStruc \colon I \to \N$
satisfy
for all $i \in I$ that $f_i \in C(\R^{\inputStruc(i)}, \R^{\outputStruc(i)})$
\cfload.
Then the following two statements are equivalent:
\begin{enumerate}[(i)]
\item 
\label{approximation_space_equivalenceGH:item1}
It holds that 
$f \in \ASGH_{a, \mathfrak{w}, G, H }^{r_0, r_1, \ldots, r_N}(I, \mathfrak{d})$ \cfload.

\item 
\label{approximation_space_equivalenceGH:item2}
There exist $K \in [0,\infty)$, $(\ANNf_{i,\varepsilon})_{(i, \varepsilon) \in I \times (0,1]} \subseteq \ANNs$ such that
for all $i \in I$, $\varepsilon \in (0,1]$, $x \in \R^{\inputStruc(i)}$ it holds that
\begin{equation}
\label{approximation_space_equivalenceGH:concl1}
  \param(\ANNf_{i,\varepsilon}) 
\leq
  K \varepsilon^{-r_0}
  \smallprod{l}{1}{N} |\mathfrak{d}_l(i)|^{r_l},
\qquad
    \functionANN(\ANNf_{i,\varepsilon}) \in C(\R^{\inputStruc(i)}, \R^{\outputStruc(i)} ),
\end{equation}
\begin{equation}
\label{approximation_space_equivalenceGH:concl3}
    \mathfrak{w}( \norm{ x })
    \norm{ f_i(x) - ( \functionANN(\ANNf_{i,\varepsilon}) ) (x) }  \leq \varepsilon,
\qandq
	\norm{( \functionANN(\ANNf_{i,\varepsilon}) ) (x)} \leq G_{i,\varepsilon} (H_{i,\varepsilon} + \norm{x})
\end{equation}
\cfload.
\end{enumerate}
\end{lemma}

\begin{proof}%
Note that \eqref{approximation_spaces_GH:eq1} establishes that (\eqref{approximation_space_equivalenceGH:item1} $\Leftrightarrow$ \eqref{approximation_space_equivalenceGH:item2}).
The proof of \cref{approximation_space_equivalenceGH} is thus complete.
\end{proof}

\begin{lemma}[Characterization of approximation spaces without linear growth conditions]
\cfclear
\label{approximation_space_equivalence}
Let $\mathfrak{d} = (\mathfrak{d}_1, \mathfrak{d}_2, \ldots, \mathfrak{d}_N)$ be an N-fold dimension mapping\cfadd{dimension_mapping_1} on $I$,
let $r_0, r_1, \ldots, r_N \in \R$, $a \in C(\R, \R)$, $f = (f_i)_{i \in I} \in \setofFF_I$,
and let $\mathfrak{w} \colon [0, \infty) \to [0, \infty)$,
$\inputStruc \colon I \to \N$, and $\outputStruc \colon I \to \N$
satisfy
for all $i \in I$ that $f_i \in C(\R^{\inputStruc(i)}, \R^{\outputStruc(i)})$
\cfload.
Then the following two statements are equivalent:
\begin{enumerate}[(i)]
\item 
\label{approximation_space_equivalence:item1}
It holds that 
$f \in \ASrates{a, \mathfrak{w}}{r_0, r_1, \ldots, r_N}(I, \mathfrak{d})$ \cfload.

\item 
\label{approximation_space_equivalence:item2}
There exist $K \in [0,\infty)$, $(\ANNf_{i,\varepsilon})_{(i, \varepsilon) \in I \times (0,1]} \subseteq \ANNs$ such that
for all $i \in I$, $\varepsilon \in (0,1]$ it holds that
\begin{equation}
\label{approximation_space_equivalence:concl1}
  \param(\ANNf_{i,\varepsilon}) 
\leq
  K \varepsilon^{-r_0}
  \smallprod{l}{1}{N} |\mathfrak{d}_l(i)|^{r_l},
\qquad
    \functionANN(\ANNf_{i,\varepsilon}) \in C(\R^{\inputStruc(i)}, \R^{\outputStruc(i)} ),
\end{equation}
\begin{equation}
\label{approximation_space_equivalence:concl3}
\andq
    \sup\nolimits_{x \in \R^{\inputStruc(i)}} 
    \big( \mathfrak{w}( \norm{ x })
    \Norm{ f_i(x) - ( \functionANN(\ANNf_{i,\varepsilon})) (x) }  \big)  
\leq 
	\varepsilon
\end{equation}
\cfload.
\end{enumerate}
\end{lemma}

\begin{proof}%
Note that \cref{ASRates_deflemma} establishes that (\eqref{approximation_space_equivalence:item1} $\Leftrightarrow$ \eqref{approximation_space_equivalence:item2}).
The proof of \cref{approximation_space_equivalence} is thus complete.
\end{proof}

\cfclear
\begin{lemma}[Inclusions between approximation spaces]
\label{inclusions}
Let $\mathfrak{d} = (\mathfrak{d}_1, \mathfrak{d}_2, \ldots, \mathfrak{d}_N)$ be an N-fold dimension mapping\cfadd{dimension_mapping_1} on $I$,
let 
	$r_0, r_1, \ldots, r_N, \mathbf{r}_0, \mathbf{r}_1, \ldots, \mathbf{r}_N \in \R$, 
	$a \in C(\R, \R)$, 
	$G = (G_{i,\varepsilon})_{(i,\varepsilon) \in I\times (0,1]} \subseteq (0, \infty]$, 
	$\mathbf{G} = (\mathbf{G}_{i,\varepsilon})_{(i,\varepsilon) \in I\times (0,1]} \subseteq (0, \infty]$, 
	$H = (H_{i,\varepsilon})_{(i,\varepsilon) \in I\times (0,1]} \subseteq (0, \infty]$,
	$\mathbf{H} = (\mathbf{H}_{i,\varepsilon})_{(i,\varepsilon) \in I\times (0,1]} \subseteq (0, \infty]$
satisfy for all 
	$i \in I$, $\varepsilon \in (0,1]$
that
$
	r_i \leq \mathbf{r}_i
$,
$
	G_{i,\varepsilon} \leq \mathbf{G}_{i,\varepsilon}
$,
and
$
	H_{i,\varepsilon} \leq \mathbf{H}_{i,\varepsilon}
$,
and let $\mathfrak{w} \colon [0, \infty) \to [0, \infty)$ be a function \cfload. 
Then
\begin{equation}
\label{Inclusion_Appr}
    \ASGH_{a, \mathfrak{w}, G, H }^{r_0, r_1, \ldots, r_N}(I,\mathfrak{d})
\subseteq
	\ASGH_{a, \mathfrak{w}, \mathbf{G}, \mathbf{H} }^{r_0, r_1, \ldots, r_N}(I,\mathfrak{d})
\subseteq
	\ASGH_{a, \mathfrak{w}, \mathbf{G}, \mathbf{H} }^{\mathbf{r}_0, \mathbf{r}_1, \ldots, \mathbf{r}_N}(I,\mathfrak{d})
\subseteq
     \ASnorates_{a, \mathfrak{w} }(I,\mathfrak{d})
\end{equation}
\cfout.
\end{lemma}

\begin{proof}%
Note that
\enum{ 
\eqref{approximation_spaces_GH:eq1};
 \eqref{approximation_spaces_norates:eq1}  
}  
establish \eqref{Inclusion_Appr}.
The proof of \cref{inclusions} is thus complete.
\end{proof}

\subsection{Linear combinations in approximation spaces}
\label{subsect:linear_comb}

\cfclear
\begin{lemma}[Linear combinations in approximation spaces]
\label{algebra_lemma_1}
Let $\mathfrak{d} = (\mathfrak{d}_1, \mathfrak{d}_2, \ldots, \mathfrak{d}_N)$ be an $N$-fold dimension mapping\cfadd{dimension_mapping_1} on $I$,
let $\lambda, c \in \R$, $a \in C(\R, \R)$, $(\mathfrak{i}_d)_{d \in \N} \subseteq \N$, 
$(\ANNi_d)_{d \in \N} \subseteq \ANNs$ satisfy for all $d \in \N$ that
   $\functionANN(\ANNi_d) = \id_{\R^d}$
    $\dims(\ANNi_d) = (d,\mathfrak{i}_d,d)$, and
    $\mathfrak{i}_d \leq c d$, %
let $\mathfrak{w} \colon [0, \infty) \to [0, \infty)$, $\inputStruc \colon I \to \N$, and $\outputStruc \colon I \to \N$ be functions,
and let $f = (f_i)_{i \in I}$, $g = (g_i)_{i \in I} \in \ASnorates_{a, \mathfrak{w}}(I, \mathfrak{d})$ 
satisfy for all $i \in I$ that 
$
    \{ f_i , g_i \} \subseteq C(\R^{\inputStruc(i)}, \R^{\outputStruc(i)} )
$
\cfload.
Then 
\begin{equation}
\label{algebra_lemma_1:concl}
    (I \ni i \mapsto (\lambda f_i + g_i) \in \mycup_{k,l \in \N} C(\R^k, \R^l)) \in \ASnorates_{a, \mathfrak{w}}(I, \mathfrak{d}).
\end{equation}
\end{lemma}
\begin{proof}%
\cref{algebra_lemma_1} follows from \cref{Lemma:SumsOfANNS}. See \cite[Lemma 3.17]{Beneventano2020v1} for more details.
\end{proof}

\subsection{Limits in approximation spaces}
\label{subsect:limits}

\cfclear
\begin{prop}[Limits in approximation spaces]
\label{lim_approximation}
Let $\mathfrak{d} = (\mathfrak{d}_1, \mathfrak{d}_2, \ldots, \mathfrak{d}_N)$ be an N-fold dimension mapping\cfadd{dimension_mapping_1} on $I$,
let $\mathfrak{w} \colon [0, \infty) \to [0, \infty)$, $\inputStruc \colon I \to \N$, and $\outputStruc \colon I \to \N$ be functions, 
let $R \in (0, \infty)$, $ r_0,r_1, \ldots, r_N, \alpha_1, \alpha_2, \ldots, \alpha_N, \rho \in [0,\infty)$, $a \in C(\R, \R)$, 
$ f = (f_{i,n})_{(i,n) \in I \times \N} \in
\ASrates{a, \mathfrak{w}}{r_0, r_1, \ldots, r_N, \rho}(I \times \N, \mathfrak{d} \proddimmap \id_\N)$,
$g = (g_i)_{i \in I} \in \setofFF_I$
satisfy for all $i \in I$, $n \in \N$ that
$\{f_{i,n}, g_{i}\} \subseteq C(\R^{\inputStruc(i)}, \R^{\outputStruc(i)})$,
and assume
\begin{equation}
\label{lim_approximation:hyp2}
    \sup_{i \in I} \sup_{n \in \N} \sup_{x \in     \R^{\inputStruc(i)}} 
    \left(
    \frac{ 
      \mathfrak{w}( \norm{ x })
      \norm{ g_i(x) - f_{i,n}(x) }
    }{
    n^{-R} \prod_{l= 1}^N |\mathfrak{d}_l(i)|^{\alpha_l}
    } \right) < \infty
\end{equation}
\cfload.
Then 
$
  g
\in 
  \ASrates{a, \mathfrak{w}}{r_0 + \rho/R, r_1 + \rho \alpha_1/R , \ldots, r_N + \rho \alpha_N/R}(I, \mathfrak{d})
$.
\end{prop}

\begin{proof}%
\label{lim_approximation_proof}
Throughout this proof let $D \in [0,\infty)$ satisfy
\begin{equation}
\label{lim_approximation_proof:not_1}
    D = \sup_{i \in I} \sup_{n \in \N} \sup_{x \in \R^{\inputStruc(i)}} 
    \left( \frac{ 
      \mathfrak{w}( \norm{ x })
      \norm{ g_i(x) - f_{i,n}(x) }
    }{
    n^{-R} \smallprod{l}{1}{N} |\mathfrak{d}_l(i)|^{\alpha_l}
    }\right),
\end{equation}
let 	$(\mathfrak{q}_{i, \delta})_{(i, \delta) \in I \times (0, 1]} \subseteq \R$,  
	$(\mathfrak{n}_{i, \delta})_{(i, \delta) \in I \times (0, 1]} \subseteq \N$
satisfy for all 
	$i \in I$, $\delta \in (0,1]$ 
that 
\begin{equation}
\label{lim_approximation_proof:not_3.0}
    \mathfrak{q}_{i, \delta} = \big( 
    D \delta^{-1}
    \smallprod{l}{1}{N}|\mathfrak{d}_l(i)|^{\alpha_l}
     \big)^{\!\!\nicefrac{1}{R}}
\qandq
	 \mathfrak{n}_{i, \delta} = \min(\N \cap [\mathfrak{q}_{i, \delta}, \infty)).
\end{equation}
Note that \eqref{lim_approximation_proof:not_1} assures that for all $i \in I$, $n \in \N$ it holds that 
\begin{equation}
\label{lim_approximation_proof:not_2}
   \sup\nolimits_{x \in \R^{\inputStruc(i)}} 
    \big(
      \mathfrak{w}( \norm{ x })
      \norm{ g_i(x) - f_{i,n}(x) } \!
     \big)
     \leq 
     D n^{-R}
     \smallprod{l}{1}{N}|\mathfrak{d}_l(i)|^{\alpha_l}.
\end{equation}
This ensures that for all $i \in I$, $\delta \in (0,1]$ it holds that
\begin{equation}
\label{lim_approximation_proof:not_4}
\begin{split}
    \sup\nolimits_{x \in \R^{\inputStruc(i)}} 
    \big(
    \mathfrak{w}( \Norm{ x })
    \Norm{ g_i(x) - f_{i,\mathfrak{n}_{i, \delta}}(x) }
    \big) 
&\leq
    D | \mathfrak{n}_{i, \delta}|^{-R}
    \smallprod{l}{1}{N}|\mathfrak{d}_l(i)|^{\alpha_l} \\
&\leq
    D |\mathfrak{q}_{i, \delta}|^{-R} 
    \smallprod{l}{1}{N}|\mathfrak{d}_l(i)|^{\alpha_l}
= 
    \delta.
\end{split}
\end{equation}
In addition, observe that \cref{approximation_space_equivalence} and the assumption that $f \in \ASrates{a, \mathfrak{w}}{r_0, r_1, \ldots, r_N, \rho}
(I \times \N, \mathfrak{d} \proddimmap \id_\N)$ imply that there exist $K \in [0,\infty)$, $(\ANNf_{i,n,\varepsilon})_{(i, n, \varepsilon) \in I \times \N \times (0,1]} \subseteq \ANNs$ which satisfy for all $i \in I$, $n \in \N$, $\varepsilon \in (0,1]$ that
\begin{equation}
\label{lim_approximation_proof:not_5}
  \param(\ANNf_{i,n,\varepsilon}) 
\leq
  K \varepsilon^{-r_0} n^{\rho}
   \smallprod{l}{1}{N}  |\mathfrak{d}_l(i)|^{r_l} ,
   \qquad
    \functionANN(\ANNf_{i,n,\varepsilon}) \in C(\R^{\inputStruc(i)}, \R^{\outputStruc(i)} ),
\end{equation}
\begin{equation}
\label{lim_approximation_proof:not_7}
  \andq
  \sup\nolimits_{x \in \R^{\inputStruc(i)}} 
      \big( \mathfrak{w}( \norm{ x })
      \norm{ f_{i,n}(x) - \functionANN(\ANNf_{i,n,\varepsilon})(x) } \! \big)
 \leq
  \varepsilon
\end{equation}
\cfload.
In the next step let $(\ANNh_{i, \delta})_{(i,\delta) \in I \times (0, 1]}\subseteq \ANNs$ satisfy for all $i \in I$, $\delta \in (0, 1]$ that
\begin{equation}
\label{lim_approximation_proof:not_8}
    \ANNh_{i, \delta} = \ANNf_{i,(\mathfrak{n}_{i,{\delta}/{2}}),{\delta}/{2}}.
\end{equation}
Note that \eqref{lim_approximation_proof:not_5} and \eqref{lim_approximation_proof:not_8} imply that for all $i \in I$, $\delta \in (0, 1]$ it holds that
\begin{equation}
\label{lim_approximation_proof:dis_1}
    \param(\ANNh_{i, \delta})
    \leq 
    K \big[\tfrac{\delta}{2}\big]^{-r_0}
    \big| \mathfrak{n}_{i,{\delta}/{2}}\big|^{\rho}
     \smallprod{l}{1}{N}  |\mathfrak{d}_l(i)|^{r_l} .
\end{equation}
Next observe that \eqref{lim_approximation_proof:not_3.0} 
implies that for all $i \in I$, $\delta \in (0, 1]$ it holds that
\begin{equation}
    \mathfrak{n}_{i, \delta} \leq
    \big( D \delta^{-1}
    	\smallprod{l}{1}{N}  |\mathfrak{d}_l(i)|^{\alpha_l} 
    \big)^{ \!\nicefrac{1}{R}} +1.
\end{equation}
This, \eqref{lim_approximation_proof:dis_1}, and
the fact that for all $a, b \in \R$, $\eta \in [0, \infty)$ it holds that $|a+b|^\eta \leq 2^{\max\{\eta-1,0\} }(|a|^\eta + |b|^\eta)$ ensure that for all $i \in I$, $\delta \in (0, 1]$ it holds that
\begin{equation}
\begin{split}
    &\param(\ANNh_{i, \delta}) \\
& \leq 
    K \left(\tfrac{\delta}{2}\right)^{\!\!-r_0}
    2^{\max\{\rho-1,0\}} \Big[ \big( 
    D2 \delta^{-1}
    	\smallprod{l}{1}{N}  |\mathfrak{d}_l(i)|^{\alpha_l} 
     \big) ^{\nicefrac{\rho}{R}} + 1 \Big]
       \smallprod{l}{1}{N}  |\mathfrak{d}_l(i)|^{r_l} 
    \\
    & = 2^{\max\{\rho-1,0\} + r_0 + \frac{\rho}{R}} 
    KD^{\frac{\rho}{R}}  \delta ^{-r_0-\frac{\rho}{R}}
        \big[ \smallprod{l}{1}{N}  |\mathfrak{d}_l(i)|^{r_l+\frac{\alpha_l\rho}{R}} \big] 
     +  2^{\max\{\rho-1,0\} + r_0 }  K  \delta^{-r_0} 
      \smallprod{l}{1}{N}  |\mathfrak{d}_l(i)|^{r_l} .
\end{split}
\end{equation}
This, the fact that for all $l \in \{1, \ldots, N\}$ it holds that $r_l \leq r_l+\frac{\alpha_l\rho}{R}$, and the fact that $r_0 \leq r_0+\frac{\rho}{R}$ imply that for all $i \in I$, $\delta \in (0,1]$ it holds that
\begin{equation}
    \label{lim_approximation_proof:con_1}
    \param(\ANNh_{i, \delta}) \leq
    2^{\max\{\rho,1\} + r_0 + \frac{\rho}{R}}  K (\max\{1, D\}) ^{\frac{\rho}{R}}
    \delta ^{-r_0-\frac{\rho}{R}}
    \smallprod{l}{1}{N}  |\mathfrak{d}_l(i)|^{r_l+\frac{\alpha_l\rho}{R}} .
\end{equation}
Moreover, note that \eqref{lim_approximation_proof:not_5} and \eqref{lim_approximation_proof:not_8} assure that for all $i \in I$, $\delta \in (0, 1]$ it holds that
\begin{equation}
\label{lim_approximation_proof:con_2}
    \functionANN(\ANNh_{i, \delta}) \in C(\R^{\inputStruc(i)}, \R^{\outputStruc(i)} ).
\end{equation}
In addition, note that the triangle inequality assures that for all $i \in I$, $\delta \in (0, 1]$, $x \in \R^{\inputStruc(i)}$ it holds that
\begin{equation}
\label{lim_approximation_proof:step_1}
      \Normm{ g_i(x) -  (\functionANN(\ANNh_{i, \delta}))(x) }
      \leq
      \Normm{ g_i(x) -  f_{i,\mathfrak{n}_{i,{\delta}/{2}}}(x) } + \Normm{ f_{i,\mathfrak{n}_{i,{\delta}/{2}}}(x) -  \big(\functionANN( \ANNf_{i,(\mathfrak{n}_{i,{\delta}/{2}}),{\delta}/{2}})\big)(x) }.
\end{equation}
This,
\eqref{lim_approximation_proof:not_4}, and \eqref{lim_approximation_proof:not_7} imply that for all $i \in I$, $\delta \in (0, 1]$ it holds that
\begin{equation}
\label{lim_approximation_proof:step_2}
    \begin{split}
      & \sup\nolimits_{x \in \R^{\inputStruc(i)}} 
     \big( \mathfrak{w}( \norm{ x })
      \norm{ g_i(x) -  (\functionANN(\ANNh_{i, \delta}))(x) } \! \big)
      \\ & \leq
      \big( \sup\nolimits_{x \in \R^{\inputStruc(i)}} 
      \big( \mathfrak{w}( \norm{ x })
      \Norm{ g_i(x) -  f_{i,(\mathfrak{n}_{i,{\delta}/{2}})}(x) } \big) \big) \\
      & \quad +
      \big( \sup\nolimits_{x \in \R^{\inputStruc(i)}} \big(
      \mathfrak{w}( \norm{ x })
      \Norm{ f_{i,(\mathfrak{n}_{i,{\delta}/{2}})}(x) -  \big(\functionANN( \ANNf_{i,(\mathfrak{n}_{i,{\delta}/{2}}),{\delta}/{2}})\big)(x) }\big) \big)\\
      &\leq
        \tfrac{\delta}{2} + \tfrac{\delta}{2} = \delta.
    \end{split}
\end{equation}
This, \eqref{lim_approximation_proof:con_1}, \eqref{lim_approximation_proof:con_2}, and \cref{approximation_space_equivalence} establish that 
$
  g
\in 
  \ASrates{a, \mathfrak{w}}{r_0 + \rho/R, r_1 + \rho \alpha_1/R , \ldots, r_N + \rho \alpha_N/R}(I, \mathfrak{d})
$.
The proof of \cref{lim_approximation} is thus complete.
\end{proof}

\cfclear
\begin{cor}[Sufficient conditions for approximation spaces]
\label{epsilon_approximation}
Let $\mathfrak{d} = (\mathfrak{d}_1, \mathfrak{d}_2, \ldots, \mathfrak{d}_N)$ be an N-fold dimension mapping\cfadd{dimension_mapping_1} on $I$,
let 
	$R \in (0, \infty)$, 
	$ r_0,r_1, \ldots, r_N, \alpha_1, \alpha_2, \ldots, \alpha_N, K \in [0,\infty)$, 
	$a \in C(\R, \R)$, 
	$f = (f_i)_{i \in I} \in \setofFF_I$,
	$(\ANNf_{i,\varepsilon})_{(i, \varepsilon) \in I \times (0,1]} \subseteq \ANNs$,
let
$\mathfrak{w} \colon [0, \infty) \to [0, \infty)$, $\inputStruc \colon I \to \N$, and $\outputStruc \colon I \to \N$ be functions,
assume for all $i \in I$ that $f_i \in C(\R^{\inputStruc(i)}, \R^{\outputStruc(i)})$,
and assume for all $i \in I$, $\varepsilon \in (0,1]$ that
\begin{equation}
\label{epsilon_approximation:set1}
    \param(\ANNf_{i,\varepsilon}) 
    \leq
    K
    \varepsilon^{-r_0}
    \smallprod{l}{1}{N}  |\mathfrak{d}_l(i)|^{r_l},
    \qquad
    \functionANN(\ANNf_{i,\varepsilon}) \in C(\R^{\inputStruc(i)}, \R^{\outputStruc(i)} ),
\end{equation}
\begin{equation}
\label{epsilon_approximation:set3}
    \andq
    \sup\nolimits_{x \in \R^{\inputStruc(i)}} \big(
    \mathfrak{w}( \norm{ x })
    \Norm{ f_i(x) - ( \functionANN(\ANNf_{i,\varepsilon}) ) (x) } 
    \big)  \leq K \varepsilon^{R}
     \smallprod{l}{1}{N}  |\mathfrak{d}_l(i)|^{\alpha_l}
\end{equation}
\cfload.
Then
$
    f \in \ASrates{a, \mathfrak{w}}{\frac{r_0}{R}, r_1 + \frac{r_0\alpha_1}{R}, \ldots, r_N + \frac{r_0\alpha_N}{R}}(I, \mathfrak{d})
$
\cfout.
\end{cor}

\begin{proof}%
\cref{epsilon_approximation} follows from \cref{lim_approximation}. See \cite[Corollary 3.19]{Beneventano2020v1} for more details.
\end{proof}

\subsection{Infinite compositions in approximation spaces}
\label{subsect:compositions}

\newcommand{\compNumber}{\mathfrak{n}}
\cfclear
\begin{prop}[Infinite composition of function families in approximation spaces]
\label{composition_statement}
Let $\mathfrak{d} = (\mathfrak{d}_1, \mathfrak{d}_2, \ldots, \mathfrak{d}_N)$ be an N-fold dimension mapping\cfadd{dimension_mapping_1} on $I$,
 let $\mathsf{d} = (\mathsf{d}_1, \mathsf{d}_2, \ldots,  \mathsf{d}_M)$ be an M-fold dimension mapping\cfadd{dimension_mapping_1} on $J$,
let 
	$\kappa, \allowbreak 
	r_0, r_1, \ldots, r_N, \allowbreak 
	\mathbf{r}_1, \mathbf{r}_2, \ldots, \mathbf{r}_M, \allowbreak 
	\alpha_1, \alpha_2, \ldots, \alpha_M, \allowbreak 
	\beta_1, \beta_2, \ldots, \beta_N, \allowbreak 
	\gamma_1, \gamma_2, \ldots, \gamma_M, \allowbreak 
	R_1, R_2, \ldots, R_N, \allowbreak \mathbf{R}_1, \mathbf{R}_2, \ldots, \mathbf{R}_M \in [0,\infty)$, 
	$\rho \in [0,1)$,
	$(\compNumber_j)_{j \in J} \subseteq \{2, 3, 4, \ldots \}$,
	$A = \{(i,j,k) \in I \times J \times \N \colon k \leq \compNumber_j \}$, 
	$B = \{(i,j,k) \in A \colon k \leq \compNumber_j-1 \}$, 
	$(L^{j,k}_i)_{(i, j, k) \in A } \subseteq (0,\infty]$,  
	$G = (G^{j,k}_{i, \varepsilon})_{(i, j, k, \varepsilon) \in B \times (0,1]} \subseteq (0,\infty]$,
	$H = (H^{j,k}_{i, \varepsilon})_{(i, j, k, \varepsilon) \in B \times (0,1]} \subseteq (0,\infty]$,
	$a \in C(\R, \R)$
satisfy for all 
	$i \in \{1, 2, \ldots, N\}$, 
	$j \in \{1, 2, \ldots, M\}$ 
that 
	$R_i = 2r_i+\frac{2r_0\beta_i}{1 - \rho}$ 
and
	$\mathbf{R}_j = 2\mathbf{r}_j+\alpha_j+\frac{2r_0((\kappa+1)\alpha_j+\gamma_j)}{1 - \rho}$, 
assume
\begin{equation}
\label{composition_statement:ass1}
\begin{split}
  &\left[
    \sup_{(i, j, k) \in A}
    \sup_{ \varepsilon \in (0,1]}
    \frac{
        \big[ \prod_{l = k+1}^{\compNumber_j} L^{j,l}_i \big]
        \left[
          \max \!\big(\{H^{j,l}_{i, \varepsilon} \in (0, \infty] \colon l \in \N \cap (0,k)\} \cup \{ 1 \} \big)
          \big[ \prod_{l = 1}^{k-1} \max\{G^{j,l}_{i, \varepsilon}, 1 \} \big]
        \right]^\kappa
    }{
      \varepsilon^{-\rho}
      \big[ \smallprod{l}{1}{N}  |\mathfrak{d}_l(i)|^{\beta_l} \big]
      \big[ \smallprod{l}{1}{M}  |\mathsf{d}_l(j)|^{\gamma_l} \big]
    }
  \right] \\
&\quad
  +
    \sup_{j \in J}
    \frac{\compNumber_j}{
      \smallprod{l}{1}{M}  |\mathsf{d}_l(j)|^{\alpha_l}
    }
   < \infty,
\end{split}
\end{equation}
let 
	$\mathfrak{w} \colon [0, \infty) \to [0, \infty)$,
	$\inputStruc \colon A \to \N$, and 
	$\outputStruc \colon A \to \N$ 
satisfy
for all 
	$v \in [0,\infty)$, $(i,j,k) \in B$ 
that 
$
	  \mathfrak{w}(v) = (1 + v^\kappa)^{-1}
$
and
$
	\outputStruc(i,j,k) = \inputStruc(i,j,k+1),
$
for every $j \in J$, $k \in \{1, \ldots, \compNumber_j \}$ let $f^{j,k} = (f^{j,k}_i)_{i \in I} \in \setofFF_I$
satisfy for all 
	$i \in I$,
	$x, y \in \R^{\inputStruc(i,j,k)}$  that
$f^{j,k}_i \in C(\R^{\inputStruc(i,j,k)}, \R^{\outputStruc(i,j,k)})$ and
$
  \Norm{ f^{j,k}_i(x) - f^{j,k}_i(y) }
\leq
  L^{j,k}_i \norm{ x - y }
$,
and assume
\begin{equation}
\label{composition_statement:ass5b}
\left( I \times J  \ni (i,j) \mapsto
f^{j,\compNumber_j}_i \in \mycup_{k,l \in \N} C(\R^k, \R^l) \right) \in 
\ASrates{a, \mathfrak{w}}{r_0, r_1, \ldots, r_N, \mathbf{r}_1, \mathbf{r}_2, \ldots, \mathbf{r}_M}(I \times J, (\mathfrak{d} \proddimmap \mathsf{d}))
\end{equation}
and
\begin{equation}
\label{composition_statement:ass5a}
\left( B \ni (i,j,k) \mapsto
f^{j,k}_i \in \mycup_{m,l \in \N} C(\R^m, \R^l) \right) \in 
\ASGH_{a, \mathfrak{w}, G, H}^{r_0, r_1, \ldots, r_N, \mathbf{r}_1, \mathbf{r}_2, \ldots, \mathbf{r}_M, 0}(B, (\mathfrak{d} \proddimmap \mathsf{d} \proddimmap \id_\N)|_B)
\end{equation}
\cfload.
Then
\begin{equation}
\label{composition_statement:conc}
    \big( 
    \prodofstruct\nolimits_{j \in J}  
    \big( 
    f^{j,\compNumber_j} \compofstruct f^{j,\compNumber_j-1} \compofstruct \ldots \compofstruct f^{j,1} 
    \big) \big) \in
    \ASrates{a, \mathfrak{w}}{
    \frac{2r_0}{1-\rho}, R_1, R_2, \ldots, R_N, \mathbf{R}_1, \mathbf{R}_2, \ldots, \mathbf{R}_M
    }(I \times J, \mathfrak{d} \proddimmap \mathsf{d})
\end{equation}
\cfout.
\end{prop}

\begin{proof}%
Throughout this proof let $D \in [0,\infty)$ satisfy
for all $(i,j,k) \in A$, $\varepsilon \in (0, 1]$ that
\begin{equation}
\label{composition_statement:0}
\begin{split}
    & \big[ \smallprod{l}{k+1}{\compNumber_j} L^{j,l}_i \big]
        \left[
          \max \left(\{H^{j,l}_{i, \varepsilon} \in (0, \infty] \colon l \in \N \cap (0,k)\} \cup \{ 1 \} \right)
          \big[ \smallprod{l}{1}{k-1} \max\{ G^{j,l}_{i, \varepsilon}, 1\} \big]
        \right]^\kappa
    \\ & \leq D
    \varepsilon^{-\rho}
    \big[ \smallprod{l}{1}{N}  |\mathfrak{d}_l(i)|^{\beta_l} \big]
     \big[ \smallprod{l}{1}{M}  |\mathsf{d}_l(j)|^{\gamma_l} \big]
\end{split}
\end{equation}
and
\begin{equation}
\label{composition_statement:0a}
    \compNumber_j \leq D
     \smallprod{l}{1}{M}  |\mathsf{d}_l(j)|^{\alpha_l}
\end{equation}
(cf.\ \eqref{composition_statement:ass1}).
Observe that \eqref{composition_statement:ass1} assures that for all
	$(i,j,k) \in B$,
	$\varepsilon \in (0,1]$
it holds that
\begin{equation}
\label{composition_statement:0b}
\begin{split} 
	G^{j,k}_{i, \varepsilon} < \infty,
\qquad
	H^{j,k}_{i, \varepsilon} < \infty,
\qandq
	L^{j,k+1}_{i, \varepsilon} < \infty.
\end{split}
\end{equation}
Furthermore, note that 
\cref{approximation_space_equivalenceGH}, 
\cref{approximation_space_equivalence}, 
\eqref{composition_statement:ass5b}, and \eqref{composition_statement:ass5a} ensure that there exist $K \in [0, \infty)$, 
$(\ANNf^{j,k}_{i,\varepsilon})_{(i,j,k,\varepsilon) \in A \times (0,1]} \subseteq \ANNs$ which satisfy
for all $(i,j,k) \in A$, $l \in \{1, 2, \ldots, \compNumber_j-1\}$, $\varepsilon \in (0,1]$, $x \in \R^{\inputStruc(i,j,k)}$ that
\begin{equation}
\label{composition_statement:ass5}
    \param(\ANNf^{j,k}_{i,\varepsilon}) 
    \leq
    K \varepsilon^{-r_0}
    \big[ \smallprod{l}{1}{N}  |\mathfrak{d}_l(i)|^{r_l} \big]
    \big[ \smallprod{l}{1}{M}  |\mathsf{d}_l(j)|^{\mathbf{r}_l} \big] \!,
\end{equation}
\begin{equation}
\label{composition_statement:ass6}
    \functionANN(\ANNf^{j,k}_{i,\varepsilon}) \in C(\R^{\inputStruc(i,j,k)}, \R^{\outputStruc(i,j,k)} ),
    \qquad
    \mathfrak{w}( \norm{ x })
    \Normm{ f_i^{j,k}(x) - \big( \functionANN(\ANNf^{j,k}_{i,\varepsilon}) \big) (x) }
     \leq \varepsilon, 
\end{equation}
\begin{equation}
\label{composition_statement:ass7}
	\andq
    \Normm{ \big( \functionANN(\ANNf^{j,l}_{i,\varepsilon}) \big) (x) }
    \leq G^{j,l}_{i, \varepsilon}
    \big( H^{j,l}_{i, \varepsilon} + \Norm{x} \big)
\end{equation}
\cfload.
In the next step  let $X:\{(i, j, k, \varepsilon, x) \in I \times J \times \N_0 \times (0, 1] \times \left( \mycup_{n \in \N}\R^n \right) \colon x \in \R^{\inputStruc(i,j,1)}, k \leq \compNumber_j\} \to \mycup_{n \in \N}\R^n$ satisfy for all $(i, j, k) \in A$, $\varepsilon \in (0,1]$, $x \in \R^{\inputStruc(i,j,1)}$ that
\begin{equation}
\label{composition_statement:8}
    X(i, j, 0, \varepsilon, x) = x \qandq
    X(i, j, k, \varepsilon, x) = \functionANN (\ANNf_{i, \varepsilon}^{j,k})(X(i, j, k-1, \varepsilon, x))
\end{equation}
and let $(\ANNh_{i, \varepsilon}^j)_{(i, j, \varepsilon) \in I \times J \times (0, 1]} \subseteq \ANNs$ satisfy for all $i \in I$, $j \in J$, $\varepsilon \in (0, 1]$ that
$
    \ANNh_{i, \varepsilon}^j = (\ANNf_{i, \varepsilon}^{j,\compNumber_j}) \compANNbullet \ldots \compANNbullet (\ANNf_{i, \varepsilon}^{j,1})
$
\cfload.
Observe that \eqref{composition_statement:8} and item~\eqref{PropertiesOfCompositions_n:Realization} in \cref{Lemma:PropertiesOfCompositions_n2} demonstrate that for all $i \in I$, $j \in J$, $\varepsilon \in (0, 1]$, $x \in \R^{\inputStruc(i,j,1)}$ it holds that
\begin{equation}
\label{composition_statement:3}
    \big( \functionANN (\ANNh_{i, \varepsilon}^j) \big) (x) = \big( \functionANN (\ANNf_{i, \varepsilon}^{j,\compNumber_j}) \circ \ldots \circ \functionANN (\ANNf_{i, \varepsilon}^{j,1}) \big) (x) = X(i, j, \compNumber_j, \varepsilon, x).
\end{equation}
This, 
the assumption that for all $(i,j,k) \in A$ it holds that
$f^{j,k}_i \in C(\R^{\inputStruc(i,j,k)}, \R^{\outputStruc(i,j,k)})$,
and \eqref{composition_statement:ass6} ensure that for all $i \in I$, $j \in J$, $\varepsilon \in (0, 1]$ it holds that
\begin{equation}
\label{composition_statement:2}
    \functionANN (\ANNh_{i, \varepsilon}^j)  \in C(\R^{\inputStruc(i,j,1)}, \R^{\outputStruc(i,j,\compNumber_j)}) 
\quad \text{and} \quad
    \big( f_i^{j,\compNumber_j} \circ f_i^{j,\compNumber_j-1} \circ \ldots \circ f_i^{j,1} \big)  \in C(\R^{\inputStruc(i,j,1)}, \R^{\outputStruc(i,j,\compNumber_j)}).
\end{equation}
Moreover, note that \eqref{composition_statement:0a}, \eqref{composition_statement:ass5}, and \cref{Lemma:PropertiesOfCompositions_n3} assure that for all $i \in I$, $j \in J$, $\varepsilon \in (0,1]$ it holds that
\begin{equation}
\label{composition_statement:5}
\begin{split}
    \param \big( \ANNh^{j}_{i,\varepsilon} \big)
    &= \param \big( \ANNf_{i, \varepsilon}^{j,\compNumber_j} \compANNbullet \ldots \compANNbullet \ANNf_{i, \varepsilon}^{j,1} \big)
    \leq
    2 {\textstyle\sum_{k = 1}^{\compNumber_j-1} }
    \param \big( \ANNf_{i,\varepsilon}^{j,k} \big)
    \param \big( \ANNf_{i, \varepsilon}^{j,k+1} \big)
    \\ & \leq 
    2(\compNumber_j-1) \big( K
    \big[ \smallprod{l}{1}{N}  |\mathfrak{d}_l(i)|^{r_l} \big]
    \big[ \smallprod{l}{1}{M}  |\mathsf{d}_l(j)|^{\mathbf{r}_l} \big]
    \varepsilon^{-r_0} \big)^2
    \\ & \leq
    2DK^2
    \big[ \smallprod{l}{1}{N}  |\mathfrak{d}_l(i)|^{2r_l} \big]
    \big[ \smallprod{l}{1}{M}  |\mathsf{d}_l(j)|^{2\mathbf{r}_l + \alpha_l} \big]
    \varepsilon^{-2r_0}.
\end{split}
\end{equation}
Next observe that 
the assumption that for all $i \in I$, $j \in J$, $k \in \{2, 3, \ldots, \compNumber_j\}$, $x, y \in \R^{\inputStruc(i,j,k)}$ it holds that
$
  \Norm{ f^{j,k}_i(x) - f^{j,k}_i(y) }
\leq
  L^{j,k}_i \norm{ x - y }
$,
\eqref{composition_statement:ass6}, \eqref{composition_statement:8}, and the triangle inequality imply that for all $i \in I$, $j \in J$, $k \in \{2, 3, \ldots, \compNumber_j \}$, $\varepsilon \in (0,1]$, $x \in \R^{\inputStruc(i,j,1)}$ it holds that
\begin{equation}
\label{composition_statement:9}
\begin{split}
    &\Normm{X(i, j, k, \varepsilon, x) - \big( f_i^{j,k}
    \circ f_i^{j,k-1} \circ \ldots \circ f_i^{j,1} \big) (x)}
    \\ & =
    \Normm{\big( \functionANN (\ANNf_{i, \varepsilon}^{j,k}) \big)\left( X(i, j, k-1, \varepsilon, x)\right) - f_i^{j,k} \big( \big( f_i^{j,k-1} \circ \ldots \circ f_i^{j,1}\big)(x) \big) }
    \\ & \leq
    \Normm{f_i^{j,k}\left(X(i, j, k-1, \varepsilon, x)\right) - f_i^{j,k}  \big( \big( f_i^{j,k-1} \circ \ldots \circ f_i^{j,1} \big)(x) \big)}
    \\ & \quad +
    \Normm{\big( \functionANN (\ANNf_{i, \varepsilon}^{j,k}) \big) \left( X(i, j, k-1, \varepsilon, x)\right) - f_i^{j,k}\left(X(i, j, k-1, \varepsilon, x)\right)}
    \\ & \leq
    L_i^{j, k} \Normm{X(i, j, k-1, \varepsilon, x) -   \big( f_i^{j,k-1} \circ \ldots \circ f_i^{j,1} \big) (x)}
   + \frac{\varepsilon}{
    \mathfrak{w}\left( \norm{ X(i, j, k-1, \varepsilon, x) } \right)}.
\end{split}
\end{equation}
This, \eqref{composition_statement:8}, and \eqref{composition_statement:3} ensure that for all $i \in I$, $j \in J$, $\varepsilon \in (0,1]$, $x \in \R^{\inputStruc(i,j,1)}$ it holds that
\begin{equation}
\label{composition_statement:10}
\begin{split}
    &\Normm{ \big( \functionANN(\ANNh_{i, \varepsilon}^j) \big) (x) - 
    \big( f_i^{j,\compNumber_j} \circ f_i^{j,\compNumber_j-1} \circ \ldots \circ f_i^{j,1}\big)(x)}
    \\ & = \Normm{X(i, j, \compNumber_j, \varepsilon, x) - \big( f_i^{j,\compNumber_j} \circ f_i^{j,\compNumber_j-1} \circ \ldots \circ f_i^{j,1}\big)(x)}
    \\[1.5 ex] & \leq
    L_i^{j, \compNumber_j} \Normm{X(i, j, \compNumber_j-1, \varepsilon, x) - \big( f_i^{j,\compNumber_j-1} \circ \ldots \circ f_i^{j,1}\big)(x)}
    \\ & \quad + \frac{\varepsilon}{
    \mathfrak{w}\left( \norm{ X(i, j, \compNumber_j-1, \varepsilon, x) } \right)}
    \\ & \leq \ldots 
    \\ & \leq
    \left[  \sum_{k = 0}^{\compNumber_j-2} 
    \left( \big[ \smallprod{l}{\compNumber_j - k + 1}{\compNumber_j} L_i^{j, l} \big]
    \frac{\varepsilon} {\mathfrak{w} \left( \norm{ X(i, j, \compNumber_j-k-1, \varepsilon, x) }\right) } \right) \right]
    \\ & \quad + 
    \big[ \smallprod{l}{2}{\compNumber_j} L_i^{j, l} \big] \norm{X(i, j, 1, \varepsilon, x) - f_i^{j,1}(x)}.
\end{split}
\end{equation}
In addition, note that \eqref{composition_statement:ass6} and \eqref{composition_statement:3} assure that $i \in I$, $j \in J$, $\varepsilon \in (0,1]$, $x \in \R^{\inputStruc(i,j,1)}$ it holds that
\begin{equation}
\begin{split}
    \norm{X(i, j, 1, \varepsilon, x) - f_i^{j,1}(x)}
    & = \norm{ \big(\functionANN (\ANNf_i^{j,1})\big)(x) -  f_i^{j,1}(x)}\\
    &\leq \frac{\varepsilon}{\mathfrak{w}(\norm{x})} = \frac{\varepsilon}{\mathfrak{w}(\norm{X(i, j, 0, \varepsilon, x)})}.
\end{split}
\end{equation}
This and \eqref{composition_statement:10} ensure that for all $i \in I$,$j \in J$, $\varepsilon \in (0,1]$, $x \in \R^{\inputStruc(i,j,1)}$ it holds that
\begin{equation}
\label{composition_statement:10a}
\begin{split}
    &\Normm{ \big( \functionANN(\ANNh_{i, \varepsilon}^j) \big) (x) - 
    \big( f_i^{j,\compNumber_j} \circ f_i^{j,\compNumber_j-1} \circ \ldots \circ f_i^{j,1}\big)(x)}
    \\[1.5 ex] & \leq 
    \varepsilon \sum_{k = 0}^{\compNumber_j-1} 
    \left( \big[ \smallprod{l}{\compNumber_j - k + 1}{\compNumber_j} L_i^{j, l} \big]
    \left(\mathfrak{w} \left( \norm{ X(i, j, \compNumber_j-k-1, \varepsilon, x) }\right) \right)^{-1} \right).
\end{split}
\end{equation}
Moreover, observe that \eqref{composition_statement:ass7} and \eqref{composition_statement:8} assure that for all $(i,j,k) \in B$, $\varepsilon \in (0,1]$, $x \in \R^{\inputStruc(i,j,1)}$ it holds that
\begin{equation}
\begin{split}
    \label{composition_statement:11}
    \norm{ X(i, j, k, \varepsilon, x) } 
    &\leq G^{j,k}_{i, \varepsilon} \left(
    H^{j,k}_{i, \varepsilon} +
    \norm{ X(i, j, k-1, \varepsilon, x) } \right)
    \\ & \leq 
    G^{j,k}_{i, \varepsilon}G^{j,k-1}_{i, \varepsilon}
    \left( H^{j,k-1}_{i, \varepsilon} +
    \norm{ X(i, j, k-2, \varepsilon, x) } \right)
    + G^{j,k}_{i, \varepsilon}H^{j,k}_{i, \varepsilon}
    \\ & \leq \ldots 
    \\ & \leq
    \big[ \smallprod{h}{1}{k} G^{j,h}_{i, \varepsilon} \big]
    \norm{x} +
    \sum_{l = 1}^{k} \big( H^{j,l}_{i, \varepsilon}
    \smallprod{h}{l}{k} G^{j,h}_{i, \varepsilon} \big).
\end{split}
\end{equation}
This,
the assumption that for all $v \in [0,\infty)$ it holds that
$
	  \mathfrak{w}(v) = (1 + v^\kappa)^{-1}
$, and 
\eqref{composition_statement:8}
imply that for all $i \in I $, $j \in J$, $k \in \{0, 1, \ldots \compNumber_j-1\}$, $\varepsilon \in (0,1]$, $x \in \R^{\inputStruc(i,j,1)}$ it holds that
\begin{equation}
\label{composition_statement:12}
\begin{split}
    \left( \mathfrak{w} \left( \norm{ X(i, j, k, \varepsilon, x) }\right) \right)^{-1}
& = 
	1 + \norm{ X(i, j, k, \varepsilon, x) }^\kappa \\
&\leq
    1+ \left[
    \big[ \smallprod{h}{1}{k} G^{j,h}_{i, \varepsilon} \big]
    \norm{x} +
    {\textstyle \sum_{l = 1}^{k} }\big( H^{j,l}_{i, \varepsilon}
    \smallprod{h}{l}{k} G^{j,h}_{i, \varepsilon} \big)
    \right]^\kappa.
\end{split}
\end{equation}
Combining this and \eqref{composition_statement:10a} ensures that for all $i \in I$, $j \in J$, $\varepsilon \in (0,1]$, $x \in \R^{\inputStruc(i,j,1)}$ it holds that
\begin{equation}
\label{composition_statement:13}
\begin{split}
    &\Normm{ \big( \functionANN(\ANNh_{i, \varepsilon}^j) \big) (x) 
    - \big( f_i^{j,\compNumber_j} \circ f_i^{j,\compNumber_j-1} \circ 
    \ldots \circ f_i^{j,1}\big)(x)}
    \\ & \leq
    \varepsilon
    \sum_{k = 1}^{\compNumber_j} 
    \left( \left( \mathfrak{w} \left( \norm{ X(i, j, k-1, \varepsilon, x) } \right) \right) ^{-1} \smallprod{l}{k+1}{\compNumber_j} L_i^{j,l} \right)
    \\ & \leq 
    \varepsilon\sum_{k = 1}^{\compNumber_j} 
    \left( \left( 1+ \left[ \big[ \smallprod{l}{1}{k-1}
    G^{j,l}_{i, \varepsilon} \big]
    \norm{x}
    +
    {\textstyle\sum_{l = 1}^{k-1}} \big( 
    	H^{j,l}_{i, \varepsilon}
   	 	\smallprod{h}{l}{k-1}
    	G^{j,h}_{i, \varepsilon}
    \big)
    \right]^\kappa\right)
    \smallprod{l}{k+1}{\compNumber_j} L_i^{j, l}\right).
\end{split}
\end{equation}
The fact that for all $a, b \in \R$, $\eta \in [0, \infty)$ it holds that $|a+b|^\eta \leq 2^{\max\{\eta-1,0\}}(|a|^\eta + |b|^\eta)$ 
hence demonstrate that for all $i \in I$, $j \in J$, $\varepsilon \in (0,1]$, $x \in \R^{\inputStruc(i,j,1)}$ it holds that
\begin{equation}
\label{composition_statement:14}
\begin{split}
    & \Normm{ \big( \functionANN(\ANNh_{i, \varepsilon}^j) \big) (x) - \big( f_i^{j,\compNumber_j} \circ f_i^{j,\compNumber_j-1} 
    \circ \ldots \circ f_i^{j,1}\big)(x)}
    \\[1.5  ex] & \leq 
    \varepsilon \sum_{k = 1}^{\compNumber_j}
    \left( \left( 1+  2^{\max\{\kappa-1, 0\}}
    \left( \big[ \Norm{x}  \smallprod{l}{1}{k-1}
    G^{j,l}_{i, \varepsilon}
    \big]^\kappa + \left[
    {\textstyle \sum_{l = 1}^{k-1} }\big(
    	H^{j,l}_{i, \varepsilon}
    	\smallprod{h}{l}{k-1}
    	G^{j,h}_{i, \varepsilon}
    \big)
    \right]^\kappa\right)\right)
    \smallprod{l}{k+1}{\compNumber_j} L_i^{j, l}\right)
    \\[1.5  ex] & \leq 
    \varepsilon \sum_{k = 1}^{\compNumber_j}
    \Bigg( \Bigg( 1+  2^{\max\{\kappa-1, 0\}} \Bigg( \big[  
    \Norm{x} \smallprod{l}{1}{k-1} G^{j,l}_{i, \varepsilon} 
    \big]^\kappa 
    \\ & \quad + \left[
    k \big( \! \max \big(\{H^{j,l}_{i, \varepsilon} \in (0, \infty) \colon l \in \N \cap (0,k)\} \cup \{ 1 \} \big) \big) 
    \smallprod{l}{1}{k-1} \max\{ G^{j,l}_{i, \varepsilon}, 1 \}
    \right]^\kappa\Bigg)\Bigg)
    \smallprod{l}{k+1}{\compNumber_j} L_i^{j, l}\Bigg).
\end{split}
\end{equation}
This and \eqref{composition_statement:0} ensure that for all $i \in I$, $j \in J$, $\varepsilon \in (0,1]$, $x \in \R^{\inputStruc(i,j,1)}$ it holds that
\begin{equation}
\begin{split}
    & \Normm{ \big( \functionANN(\ANNh_{i, \varepsilon}^j) \big) (x) - \big( f_i^{j,\compNumber_j} \circ f_i^{j,\compNumber_j-1} 
    \circ \ldots \circ f_i^{j,1}\big)(x)}
    \\[1.5  ex] & \leq 
    \left( \sum_{k=1}^{\compNumber_j}
    \left(1 + 2^{\max\{\kappa-1, 0\}} \left(
    \norm{x}^\kappa + 
    k^\kappa \right) \right) \right) D
    \big[ \smallprod{l}{1}{N}|\mathfrak{d}_l(i)|^{\beta_l}
    \big]
    \big[ \smallprod{l}{1}{M}|\mathsf{d}_l(j)|^{\gamma_l} \big]
    \varepsilon^{1-\rho}
    \\[1.5  ex] & \leq 
    \left( \sum_{k=1}^{\compNumber_j}
    \left(1 + 2^{\max\{\kappa-1, 0\}}  k^\kappa \left(
    \norm{x}^\kappa + 
    1 \right) \right) \right) D
    \big[ \smallprod{l}{1}{N}|\mathfrak{d}_l(i)|^{\beta_l}
    \big]\big[ \smallprod{l}{1}{M}|\mathsf{d}_l(j)|^{\gamma_l} \big]
    \varepsilon^{1-\rho}
    \\[1.5  ex] & \leq 
    \left( \sum_{k=1}^{\compNumber_j}  k^\kappa \right)
    \left(1 + 2^{\max\{\kappa-1, 0\}}\right) 
    \left( 1 + \norm{x}^\kappa \right) D
    \big[ \smallprod{l}{1}{N}|\mathfrak{d}_l(i)|^{\beta_l}
    \big]\big[ \smallprod{l}{1}{M}|\mathsf{d}_l(j)|^{\gamma_l} \big]
    \varepsilon^{1-\rho}.
\end{split}
\end{equation}
This and the assumption that for all $v \in [0,\infty)$ it holds that
$
	  \mathfrak{w}(v) = (1 + v^\kappa)^{-1}
$
assure that for all $i \in I$, $j \in J$, $\varepsilon \in (0,1]$, $x \in \R^{\inputStruc(i,j,1)}$ it holds that
\begin{equation}
\label{composition_statement:15}
\begin{split}
    & \mathfrak{w}\left( \norm{ x} \right) \Normm{ \big( \functionANN(\ANNh_{i, \varepsilon}^j) \big) (x) 
    - \big( f_i^{j,\compNumber_j} \circ f_i^{j,\compNumber_j-1} \circ \ldots \circ f_i^{j,1}\big)(x)}
    \\[1.5  ex] & \leq 
    D  \left( \sum_{k=1}^{\compNumber_j} k^\kappa \right)
    \left(1 + 2^{\max\{\kappa-1, 0\}} \right) \left( \frac{1 + \norm{x}^\kappa}{1 + \norm{x}^\kappa} \right) 
    \big[ \smallprod{l}{1}{N}|\mathfrak{d}_l(i)|^{\beta_l}
    \big]\big[ \smallprod{l}{1}{M}|\mathsf{d}_l(j)|^{\gamma_l} \big]
    \varepsilon^{1-\rho}
    \\[1.5  ex] & = D
    \left( \sum_{k=1}^{\compNumber_j}  k^\kappa \right)
    \left(1 + 2^{\max\{\kappa-1, 0\}}\right)
    \big[ \smallprod{l}{1}{N}|\mathfrak{d}_l(i)|^{\beta_l}
    \big]\big[ \smallprod{l}{1}{M}|\mathsf{d}_l(j)|^{\gamma_l} \big]
    \varepsilon^{1-\rho}.
\end{split}
\end{equation}
Note that \eqref{composition_statement:0a}, the fact that for all $n \in \N$ it holds that $\sum_{k = 1}^n k^\kappa <(n+1)^{\kappa+1} $, and the fact that for all $a, b \in \R$, $\eta \in [0, \infty)$ it holds that $|a+b|^\eta \leq 2^{\max\{\eta-1,0\}}(|a|^\eta + |b|^\eta)$ imply that for all $j \in J$ it holds that
\begin{equation}
        \sum_{k = 1}^{\compNumber_j} k^\kappa \leq (\compNumber_j+1)^{\kappa+1} \leq 2^\kappa (\compNumber_j^{\kappa+1} + 1)
         \leq 2^\kappa (1 + D^{\kappa+1}) 
        \big[ \smallprod{l}{1}{M}|\mathsf{d}_l(j)|^{(\kappa + 1) \alpha_l} \big].
\end{equation}
This and \eqref{composition_statement:15} ensure that for all $i \in I$, $j \in J$, $\varepsilon \in (0,1]$, $x \in \R^{\inputStruc(i,j,1)}$ it holds that
\begin{equation}
\begin{split}
    & \mathfrak{w}\left( \norm{ x} \right) \Normm{ \big( \functionANN(\ANNh_{i, \varepsilon}^j) \big) (x) 
    - \big( f_i^{j,\compNumber_j} \circ f_i^{j,\compNumber_j-1} \circ \ldots \circ f_i^{j,1}\big)(x)}
    \\[1.5  ex] & \leq 
    D  2^\kappa \big(1 + 2^{\max\{\kappa-1, 0\}} \big)
    \big(1 + D^{\kappa+1}\big) \!
    \big[ \smallprod{l}{1}{N}|\mathfrak{d}_l(i)|^{\beta_l}
    \big]\big[ \smallprod{l}{1}{M}|\mathsf{d}_l(j)|^{\gamma_l+(\kappa + 1) \alpha_l} \big]
    \varepsilon^{1-\rho}.
\end{split}
\end{equation}
Combining this, \eqref{composition_statement:2}, and \eqref{composition_statement:5} with \cref{epsilon_approximation}
establishes that
\begin{equation}
    \big( 
    \prodofstruct\nolimits_{j \in J}  
    \big( 
    f^{j,\compNumber_j} \compofstruct f^{j,\compNumber_j-1} \compofstruct \ldots \compofstruct f^{j,1} 
    \big) \big) \in
    \ASrates{a, \mathfrak{w}}{ \frac{2r_0}{1-\rho}, R_1, R_2, \ldots, R_N, \mathbf{R}_1, \mathbf{R}_2, \ldots, \mathbf{R}_M} 
    (I \times J, \mathfrak{d} \proddimmap \mathsf{d})
\end{equation}
 \cfload.
 The proof of \cref{composition_statement} is thus complete.
\end{proof}

\cfclear
\begin{cor}[Single composition of function families in approximation spaces]
\label{algebra_lemma_2}
Let $\mathfrak{d} = (\mathfrak{d}_1, \mathfrak{d}_2, \ldots, \mathfrak{d}_N)$ be an N-fold dimension mapping\cfadd{dimension_mapping_1} on $I$,
let 
	$\kappa, r_0, r_1, \ldots, r_N,  \beta_1, \beta_2, \ldots, \beta_N \in [0,\infty)$, 
	$\rho \in [0,1)$, 
	$(L_i)_{i \in I} \subseteq (0, \infty]$, 
	$G = (G_{i, \varepsilon})_{(i, \varepsilon) \in I \times (0,1]} \subseteq (0, \infty]$,	
	$H = (H_{i, \varepsilon})_{(i, \varepsilon) \in I \times (0,1]} \subseteq (0, \infty]$,
	$a \in C(\R, \R)$
satisfy 
\begin{equation}
\label{algebra_lemma_2:ass1}
    \sup_{i \in I}     
    \sup_{\varepsilon \in (0,1]} 
    \bigg(
    \frac{
    L_i\left[
    \max\{H_{i, \varepsilon},1\} \max\{G_{i, \varepsilon}, 1\}
    \right]^\kappa}{
    \varepsilon^{-\rho} 
    \smallprod{l}{1}{N} |\mathfrak{d}_l(i)|^{\beta_l}
    }    
    \bigg)
    < \infty,
\end{equation}
let $\mathfrak{w} \colon [0, \infty) \to [0, \infty)$ and $\inputStruc \colon I \to \N$ satisfy for all $v \in [0,\infty)$ that
$
	  \mathfrak{w}(v) = (1 + v^\kappa)^{-1}
$,
and let 
$f \in \ASGH_{a, \mathfrak{w}, G, H}^{r_0, r_1, \ldots, r_N}(I, \mathfrak{d})$,
$g \in \ASrates{a, \mathfrak{w}}{r_0, r_1, \ldots, r_N}(I, \mathfrak{d})$
satisfy for all $i \in I$, $x,y \in \R^{\inputStruc(i)}$ that 
    $f_i \in C(\R^{\inputStruc(i)}, \R^{\inputStruc(i)} )$, 
    $g_i \in C(\R^{\inputStruc(i)}, \R^{\inputStruc(i)} )$,
    and
    $\norm{g_i(x)-g_i(y)} \leq L_i\norm{x-y}$
\cfload.
Then
\begin{equation}
\label{algebra_lemma_2:concl}
    (g \compofstruct f) \in \ASrates{a, \mathfrak{w}}{\frac{2r_0}{1-\rho}, 2r_1+\frac{2r_0\beta_1}{1 - \rho},  \ldots,2r_N+\frac{2r_0\beta_N}{1 - \rho}}(I, \mathfrak{d})
\end{equation}
\cfout.
\end{cor}
\begin{proof}%
Note that \cref{composition_statement} 
proves \eqref{algebra_lemma_2:concl}.
The proof of \cref{algebra_lemma_2} is thus complete.
\end{proof}

\subsection{Euler approximations in approximation spaces}
\label{subsect:Euler}

\cfclear
\begin{prop}[Euler approximations in approximation spaces]
\label{euler_AS}
Let $\mathfrak{d} = (\mathfrak{d}_1, \mathfrak{d}_2, \ldots, \mathfrak{d}_N)$ be an N-fold dimension mapping\cfadd{dimension_mapping_1} on $I$,
let 
	$c, T, r_0, r_1, \ldots, r_N, \allowbreak \iota_1, \ldots , \iota_N \in [0,\infty)$, 
	$H = \linebreak (H_{i, \varepsilon})_{(i, \varepsilon) \in I \times (0,1]} \subseteq (0, \infty]$, 
	$\mathcal{G} = (\mathcal{G}_{i, n, \varepsilon})_{(i,n,\varepsilon) \in  I \times \N\times (0,1]} \subseteq (0, \infty]$,
	$\mathcal{H} = (\mathcal{H}_{i, n, \varepsilon})_{(i,n,\varepsilon) \in I \times \N \times (0,1]} \subseteq (0, \infty]$, 
	$a \in C(\R, \R)$
satisfy for all 
	$i \in I$, $n \in \N$, $\varepsilon \in (0,1]$ 
that
$ 
	\mathcal{G}_{i, n, \varepsilon} = 1 + \frac{cT}{n}$ and
$ 
	\mathcal{H}_{i, n, \varepsilon} = H_{i, \frac{\varepsilon}{\max\{T, 1\}}}$,
let $\mathfrak{w} \colon [0, \infty) \to [0, \infty)$ and $\inputStruc \colon I \to \N$, 
let 
$
    f = (f_i)_{i \in I}
    \in 
    \ASGH_{a, \mathfrak{w}, c, H}
    ^{r_0, r_1, \ldots, r_N}(I, \mathfrak{d})
$
satisfy for all $i \in I$ that
$f_i \in C( \R^{\inputStruc(i)}, \R^{\inputStruc(i)})$,
let 
$\mathcal{E}^{(n)} = (\mathcal{E}^{(n)}_i)_{i \in I} \in \setofFF_I$, $n \in \N$, satisfy
for all $n \in \N$ that
$\mathcal{E}^{(n)}_i \in C(\R^{\inputStruc(i)}, \R^{\inputStruc(i)})$,
assume for all $i \in I$, $x, y \in \R^{\inputStruc(i)}$ that 
\begin{equation}
\label{euler_AS:ass1}
   \inputStruc(i) \leq c
      \smallprod{l}{1}{N} |\mathfrak{d}_l(i)|^{\iota_i},
\qquad
      \mathcal{E}^{(n)}_i(x) = x + \tfrac{T}{n} f_i(x),
\qandq
      \norm{ f_i(x) - f_i(y) } 
      \leq
        c \Norm{ x - y },
\end{equation}
and let $(\mathfrak{i}_d)_{d \in \N} \subseteq \N$, 
$(\ANNi_d)_{d \in \N} \subseteq \ANNs$ satisfy for all $d \in \N$ that
    $\functionANN(\ANNi_d) = \id_{\R^d}$, 
    $\dims(\ANNi_d) = (d,\mathfrak{i}_d,d)$, and
    $\mathfrak{i}_d \leq c d$
\cfload.
Then
\begin{equation}
\label{euler_AS:concl1}
    \left(I \times \N \ni (i,n) \mapsto (\mathcal{E}_i^{(n)}) \in C(\R^{\inputStruc(i)}, \R^{\inputStruc(i)}) \right) \in 
    \ASGH
    _{a, \mathfrak{w}, \mathcal{G}, \mathcal{H}}
    ^{r_0,  r_1+4\iota_1, \ldots,  r_N + 4\iota_N, 0}
    (I \times \N, \mathfrak{d} \proddimmap \id_{\N} )
\end{equation}
\cfout.
\end{prop}

\begin{proof}%
Note that 
the assumption that 
$ f    
 \in 
    \ASGH_{a, \mathfrak{w}, c, H}
    ^{r_0, r_1, \ldots, r_N}(I, \mathfrak{d})
$
and \cref{approximation_space_equivalenceGH}
ensure that there exist $K \in (0, \infty)$, $(\ANNf_{i,\varepsilon})_{(i, \varepsilon) \in I \times (0,1]} \subseteq \ANNs$ which satisfy
for all $i \in I$, $\varepsilon \in (0,1]$, $x \in \R^{\inputStruc(i)}$ that
\begin{equation}
\label{euler_AS:eq2}
  \param(\ANNf_{i,\varepsilon}) 
\leq
  K \varepsilon^{-r_0}
   \smallprod{l}{1}{N} |\mathfrak{d}_l(i)|^{r_l},
  \qquad
  \functionANN(\ANNf_{i,\varepsilon}) \in C(\R^{\inputStruc(i)}, \R^{\inputStruc(i)} ),
\end{equation}
\begin{equation}
\label{euler_AS:eq4}
    \mathfrak{w}( \norm{ x })
    \norm{ f_i(x) - ( \functionANN(\ANNf_{i,\varepsilon})) (x) }   \leq \varepsilon, 
\qandq
    \norm{ ( \functionANN(\ANNf_{i,\varepsilon}) ) (x) }
    \leq
    c ( H_{i, \varepsilon} + \norm{x} ).
\end{equation}
Furthermore, observe that \cref{Lemma:EulerWithANNS} and
\eqref{euler_AS:eq2} assure that for all $i \in I$, $n \in \N$, $\varepsilon \in (0,1]$ there exists $\ANNh_{i,n,\varepsilon} \in \ANNs$ which satisfies that 
\begin{equation}
\label{euler_AS:eq6}
    \functionANN(\ANNh_{i,n,\varepsilon}) \in C(\R^{\inputStruc(i)}, \R^{\inputStruc(i)})
    , \qquad
    \functionANN(\ANNh_{i,n,\varepsilon})
    =\id_{\R^{\inputStruc(i)}} + \tfrac{T}{n} \functionANN(\ANNf_{i,\frac{\varepsilon}{\max\{T, 1\}}}),
\end{equation}
\begin{equation}
\label{euler_AS:eq7}
\begin{split}
	\andq
    \param(\ANNh_{i,n,\varepsilon}) 
    & \leq 
    44\max \!\big\{1, c^3\big\} (\inputStruc(i))^4 
 	\paramANN(\ANNf_{i,\frac{\varepsilon}{\max\{T, 1\}}}).
\end{split}
\end{equation}
This, \eqref{euler_AS:ass1}, and \eqref{euler_AS:eq2} demonstrate that for all $i \in I$, $n \in \N$, $\varepsilon \in (0,1]$ it holds that
\begin{equation}
\label{euler_AS:eq8}
\begin{split}
    \param(\ANNh_{i,n,\varepsilon}) 
    & \leq 
 	44\max\big\{1, c^3\big\} c^4 \big[ \smallprod{l}{1}{N} |\mathfrak{d}_l(i)|^{4\iota_l} \big]  
 	K \big[\tfrac{\varepsilon}{\max\{T, 1\}}\big]^{\!-r_0} \smallprod{l}{1}{N} |\mathfrak{d}_l(i)|^{r_l}
 	\\ & \leq
 	44 \max\big\{1, c^7\big\}  K [\max\{T, 1\}]^{r_0} 
 	\varepsilon^{-r_0}  \smallprod{l}{1}{N} |\mathfrak{d}_l(i)|^{r_l+4\iota_l} .
\end{split}
\end{equation}
Moreover, observe that \eqref{euler_AS:ass1}, \eqref{euler_AS:eq4}, and \eqref{euler_AS:eq6} ensure that
for all $i \in I$, $n \in \N$, $\varepsilon \in (0,1]$ it holds that
\begin{equation}
\label{euler_AS:eq9}
\begin{split}
  &\sup\nolimits_{x \in \R^{\inputStruc(i)}} 
    \big(\mathfrak{w}( \norm{ x })
    \Normm{ \mathcal{E}^{(n)}_i(x) - \big( \functionANN(\ANNh_{i,n,\varepsilon}) \big) (x) }\big) \\
&=
  \sup\nolimits_{x \in \R^{\inputStruc(i)}} \left(
    \mathfrak{w}( \norm{ x })
    \Normm{ x + \tfrac{T}{n}f_i(x) - (x + \tfrac{T}{n}\big( \functionANN(\ANNf_{i,\frac{\varepsilon}{\max\{T, 1\}}}) \big) (x))}
    \right) \\
&=
  \tfrac{T}{n}
  \left(
    \sup\nolimits_{x \in \R^{\inputStruc(i)}} 
    \big(
      \mathfrak{w}( \norm{ x })
      \Normm{ f_i(x) - \big( \functionANN(\ANNf_{i,\frac{\varepsilon}{\max\{T, 1\}}}) \big) (x)} 
      \big)
    \right)
\leq 
	\tfrac{T\varepsilon}{n\max\{T, 1\}}
\leq
  \varepsilon.
\end{split}
\end{equation}
Furthermore, note that  \eqref{euler_AS:eq4} and \eqref{euler_AS:eq6} assure that
for all $i \in I$, $n \in \N$, $\varepsilon \in (0,1]$, $x \in \R^{\inputStruc(i)}$ it holds that
\begin{equation}
\label{euler_AS:eq10}
\begin{split}
    \Normm{ \big( \functionANN(\ANNh_{i,n,\varepsilon}) \big) (x) }
    &=
    \Normm{ x +  \tfrac{T}{n}\big( \functionANN(\ANNf_{i, \frac{\varepsilon}{\max\{T, 1\}}}) \big) (x) } 
    \leq
    \norm{ x } +  \tfrac{cT}{n} \big( H_{i, \frac{\varepsilon}{\max\{T, 1\}}} + \Norm{x} \big) 
    \\&\leq
    (1 + \tfrac{cT}{n})\big( H_{i, \frac{\varepsilon}{\max\{T, 1\}}} + \Norm{x} \big)
    \leq
    \mathcal{G}_{i, n, \varepsilon}
    \big( \mathcal{H}_{i, n, \varepsilon} + \Norm{x} \big)
    .
\end{split}
\end{equation}
Combining this,
\eqref{euler_AS:eq6},
\eqref{euler_AS:eq8}, 
\eqref{euler_AS:eq9},
and \cref{approximation_space_equivalenceGH} establishes \eqref{euler_AS:concl1}.
The proof of \cref{euler_AS} is thus complete.
\end{proof}

\section{Application to first order transport partial differential equations (PDEs)}
\label{sect:appl_to_PDEs}

In this section we illustrate how the theory on approximation spaces developed in \cref{sect:AS} can be used to prove results  on the approximation capacity of ANNs in the case of first order transport PDEs.
For this we recall in \cref{subsect:PDE_flow} an elementary and well-known Feynman-Kac-type formula for first order PDEs
and introduce a related PDE flow operation. 
Next we recall elementary and well-known error estimates for the Euler scheme for ordinary differential equations (ODEs) in \cref{subsect:euler_scheme}.
We then prove the main result of this paper on first order transport PDEs in  \cref{flow_statement} in \cref{subsect:PDE_flows_in_appr_space}.
Note that in our proof of \cref{flow_statement} we only deal with neural networks indirectly through the machinery of approximation spaces developed in \cref{sect:AS}.
Subsequently, in \cref{subsect:PDE_flows_relu} we present consequences of \cref{flow_statement} for realizations of neural networks with ReLU activation function in the case that the index set of the approximation spaces are the natural numbers $\N$.

\subsection{A PDE flow operation}
\label{subsect:PDE_flow}

\cfclear
\begin{lemma}[Existene and uniqueness of solution fields of ODEs]
\label{ODE_exist_unique}
Let $d \in \N$, $T \in (0,\infty)$, let $f \colon \R^d \to \R^d$ be locally Lipschitz continuous,
and assume
$
  \sup_{x \in \R^d} \frac{\Norm{f(x)}}{1 + \Norm{x}}
<
  \infty
$
\cfload.
Then
\begin{enumerate}[(i)]
\item \label{ODE_exist_unique:item1}
there exist unique  $X^x = (X^x_t)_{ t \in [0, T] } \in C( [0, T] , \R)$, $x \in \R^d$, which satisfy
for all $x \in \R^d$, $t \in [0, T]$ that
$
  X^x_t = x + \int_0^t f(X^x_s) \, ds
$
and

\item \label{ODE_exist_unique:item2}
it holds that 
$
	([0, T] \times \R^d \ni (t, x) \mapsto X^x_t \in \R)
\in
	C( [0, T] \times \R^d, \R)
$.
\end{enumerate}
\end{lemma}

\begin{proof}%
\cref{ODE_exist_unique} follows from standard results on ODEs (see, e.g., Teschl \cite[Theorem 2.2, Theorem 2.8, and Theorem 2.17]{Teschl12}). 
See, e.g., \cite[Lemma 4.1]{Beneventano2020v1} for more details.
\end{proof}

\cfclear
\begin{prop}[Feynman-Kac formula for first order PDEs]
\label{first_order_FC}
Let $d \in \N$, $T \in (0,\infty)$, $g \in C^1(\R^d, \R)$, $f \in C^1( \R^d , \R^d)$
and assume
$
	\sup_{x \in \R^d}   \frac{\Norm{f(x)}}{1 + \Norm{x}} 
<
	\infty
$
\cfload.
Then
\begin{enumerate}[(i)]
\item \label{first_order_FC:item1}
there exists a unique $X = (X^x_t)_{(t, x) \in [0, T] \times \R^d} \in C^1 ( [0, T] \times \R^d , \R^d)$ which satisfies 
for all $t \in [0, T]$, $x \in \R^d$ that
$
  X^x_t = x + \int_0^t f(X^x_s) \, ds,
$
\item \label{first_order_FC:item2}
there exists a unique $u \in C^1([0,T] \times \R^d, \R)$ which satisfies for all 
	$t \in [0, T]$,
	$x \in \R^d$
that
$u(0,x) = g(x)$ and
\begin{equation}
 \tfrac{\partial  u}{\partial t}  (t,x)  
=
  \tfrac{\partial u}{\partial x} (t,x)\, f(x),
\end{equation}
and
\item \label{first_order_FC:item3}
it holds 
for all 
	$t \in [0, T]$,
	$x \in \R^d$
that
$u(t, x) = g(X^x_t)$.
\end{enumerate}
\end{prop}

\begin{proof}%
\cref{first_order_FC} is a consequence of the standard Feynman-Kac formula for second order PDEs
(cf., e.g, Hairer et al. \cite[Corollary 4.17]{Hairer2015}). See, e.g., \cite[Lemma 4.2]{Beneventano2020v1} for an elementary proof of \cref{first_order_FC}.
\end{proof}

\begin{definition}[Flow operators]
\label{Def:flow_operator}
Let $d \in \N$, $T \in [0,\infty)$,
let $f \colon \R^d \to \R^d$ be locally Lipschitz continuous, and assume $ 
    \sup_{x \in \R^d} 
    \frac{\norm{f(x)}}{1 + \norm{x}}
    < \infty
$.
Then we denote by $\Flow^T_f \colon C(\R^d, \R) \to C(\R^d, \R)$ the function which satisfies 
for all $g \in C(\R^d, \R)$, $x \in \R^d$, $X \in C([0, T] , \R^d)$ with 
$ \forall \, t \in [0, T] \colon X_t = x + \int_{0}^t f(X_s)\, ds$
that
\begin{equation}
\label{flow_operator:eq1}
  \big(\Flow^T_f (g)\big)(x)
=
  g(X_T)
\end{equation}
(cf.\ \cref{ODE_exist_unique}).
\end{definition}

\newcommand{\flowstruct}{\cfadd{Def:flow_operation}\star}
\cfclear
\begin{definition}[Flow operations on function families]
\label{Def:flow_operation}
Let $T \in [0,\infty)$, 
let $f = (f_i)_{i \in I}$ and $g = (g_i)_{i \in I}$ be \cfadd{Def:functionfamily}function families on $I$
which satisfy for all $i \in I$ that
$f_i$ is locally Lipschitz continuous,
and
let $\inputStruc \colon I \to \N$ satisfy
for all $i \in I$ that
$f_i \in C(\R^{\inputStruc(i)}, \R^{\inputStruc(i)} )$,
$g_i \in C(\R^{\inputStruc(i)}, \R )$, 
and
$ 
    \sup_{x \in \R^{\inputStruc(i)}} 
    \frac{\norm{f_i(x)}}{1 + \norm{x}}
    < \infty 
$
\cfload.
Then we denote by 
$g \flowstruct_T f \in \setofFF_I$ the function family on $I$ which satisfies 
for all $i \in I$ that
\begin{equation}
\label{flow_operation:eq1}
  (g \flowstruct_T f)_i 
=
  \Flow^T_{f_i}(g_i)
\end{equation}
(cf.\ \cref{Def:set_of_functionfamilies}).
\end{definition}

\subsection{Error estimates for Euler approximations}
\label{subsect:euler_scheme}

\cfclear
\begin{lemma}[Convergence of the Euler scheme]
\label{Euler_method_bound}
Let $d \in \N$, $T,L \in [0, \infty)$, let $f \colon \R^d \to \R^d$ satisfy for all $x,y \in \R^d$ that
$
    \norm{f(x)-f(y)} \leq L \norm{x-y}
$,
let $X^{x} \in C ([0, T], \R^d)$, $x \in \R^{d}$, satisfy 
for all $t \in [0, T]$, $x \in \R^{d}$ that 
$
    X^{x}_t 
    =
    x + \int_0^t f(X^{x}_s) \, ds,
$
and let 
$\mathcal{X}^{n,x} \colon \{ 0, 1, \ldots, n \} \to \R^{d}$, $x \in \R^{d}$, $n \in \N$, satisfy
for all $n \in \N$, $k \in \{ 1, 2, \ldots, n \}$, $x \in \R^{d}$ that 
$\mathcal{X}^{n,x}_0 = x$ and
$
    \mathcal{X}^{n,x}_k
    =  
    \mathcal{X}^{n,x}_{k-1} + \tfrac{T}{n} f(\mathcal{X}^{n,x}_{k-1})
$
\cfload.
Then it holds for all $n \in \N$, $x \in \R^d$ that
\begin{equation}
\label{Euler_bound:concl}
    \norm{X^{x}_T - \mathcal{X}^{n,x}_n}
    \leq
    \tfrac{1}{n}
    \big( \max\{L, 1\}  L  T^2 (T+1)e^{2LT} \max\{\norm{f(0)}, 1\} \big) (1 + \norm{x}).
\end{equation}
\end{lemma}

\begin{proof}%
This is an elementary result and the proof is omitted.
See, e.g, \cite[Lemma 4.5]{Beneventano2020v1} for a detailed proof.
\end{proof}

\subsection{PDE flows in approximation spaces}
\label{subsect:PDE_flows_in_appr_space}

\cfclear
\begin{theorem}[PDE flows in approximation spaces]
\label{flow_statement}
Let $\mathfrak{d} = (\mathfrak{d}_1, \mathfrak{d}_2, \ldots, \mathfrak{d}_N)$ be an N-fold dimension mapping\cfadd{dimension_mapping_1} on $I$,
let 
	$c, T, r_0, r_1, \ldots, r_N, \allowbreak \alpha_1, \alpha_2, \ldots, \alpha_N, \allowbreak \beta_1, \beta_2, \ldots, \beta_N, \allowbreak \iota_1,\iota_2, \ldots, \iota_N, \allowbreak R_0, R_1, \linebreak \ldots, \allowbreak R_N \in [0,\infty)$, 
	$\rho \in [0,1)$, 
	$\kappa \in [1,\infty)$, 
	$H = (H_{i, \varepsilon})_{(i, \varepsilon) \in I \times (0,1]} \subseteq (0, \infty]$, 
	$a \in C(\R, \R)$,
	$(\mathfrak{i}_d)_{d \in \N} \subseteq \N$, 
	$(\ANNi_d)_{d \in \N} \subseteq \ANNs$
satisfy for all 
	$k \in \{1, 2, \ldots, N\}$,
	$d \in \N$
that 
$R_0 = 1 + \frac{2r_0 (\kappa + 2)}{1 - \rho}$,
$R_k = 
\alpha_k + 2r_k + 8\iota_k + \frac{2r_0(\alpha_k(\kappa+1)+\beta_k)}{1-\rho}
$,
 $\functionANN(\ANNi_d) = \id_{\R^d}$, 
 $\dims(\ANNi_d) = (d,\mathfrak{i}_d,d)$, and
 $\mathfrak{i}_d \leq c d$,
let $\inputStruc \colon I \to \N$ and $\mathfrak{w} \colon [0, \infty) \to [0, \infty)$ satisfy for all $v \in [0,\infty)$ that
	$\mathfrak{w}(v) = (1 + v^\kappa)^{-1}$,
let 
$
    f = (f_i)_{i \in I}
    \in 
    \ASGH_{a, \mathfrak{w}, c, H}
    ^{r_0, r_1, \ldots, r_N}(I, \mathfrak{d})
$,
$
    g = (g_i)_{i \in I}
    \in 
    \ASrates{a, \mathfrak{w}}{
    r_0, r_1, \ldots, r_N
    }(I, \mathfrak{d})
$,
and assume for all $i \in I$, $x, y \in \R^{\inputStruc(i)}$ that 
$ f_i \in C( \R^{\inputStruc(i)}, \R^{\inputStruc(i)})$, 
$ g_i \in C( \R^{\inputStruc(i)}, \R)$, 
$
\max \{ 
    \norm{ f_i(x) - f_i(y) } 
    , 
    \left| g_i(x) - g_i(y) \right|  
  \}
\leq
  c\Norm{ x - y }$, and
\begin{equation}
\label{flow_statement:ass1}
\sup_{j \in I} \sup_{\varepsilon \in (0,1]}
    \left(
	    \frac{
	    |H_{j, \varepsilon}|^\kappa
	    }{ \varepsilon^{-\rho}
	    \smallprod{l}{1}{N} |\mathfrak{d}_l(j)|^{\beta_l}} 
    	+
		\frac{\norm{f_j(0)} }{  
		\smallprod{l}{1}{N} |\mathfrak{d}_l(j)|^{\alpha_l}}
		+
		\frac{\inputStruc(j)}{\smallprod{l}{1}{N} |\mathfrak{d}_l(j)|^{\iota_l}}
	\right)
<
  \infty
\end{equation}
\cfload.
Then
$
    (g \flowstruct_T f) \in
    \ASrates{a, \mathfrak{w}}
    {R_0,  R_1,  R_2, \ldots,  R_N }
    (I, \mathfrak{d})
$
\cfout.
\end{theorem}

\begin{proof}%
Throughout this proof let 
$\mathcal{E}^{(n)} = (\mathcal{E}^{(n)}_i)_{i \in I} \in \setofFF_I$, $n \in \N$, satisfy
for all $n \in \N$, $i \in I$, $x \in \R^{\inputStruc(i)}$ that
\begin{equation}
\label{flow_statement:setting1}
  \mathcal{E}^{(n)}_i \in C(\R^{\inputStruc(i)}, \R^{\inputStruc(i)})
\qandq
  \mathcal{E}^{(n)}_i(x) = x + \tfrac{T}{n} f_i(x)
\end{equation}
\cfload,
let $X^{i,x} \in C([0, T],\R^{\inputStruc(i)})$, $x \in \R^{\inputStruc(i)}$, $i \in I$, satisfy 
for all $i \in I$, $t \in [0, T]$, $x \in \R^{\inputStruc(i)}$ that 
\begin{equation}
\label{flow_statement:settinga}
  X^{i,x}_t 
=
  x + \int_0^t f_i(X^{i,x}_s) \, ds
\end{equation}
(cf.\ \cref{ODE_exist_unique}), let 
$\mathcal{X}^{i,n,x} \colon \{ 0, 1, \ldots, n \} \to \R^{\inputStruc(i)}$, $x \in \R^{\inputStruc(i)}$, $n \in \N$, $i \in I$, satisfy
for all $n \in \N$, $k \in \{ 1, 2, \ldots, n \}$, $i \in I$, $x \in \R^{\inputStruc(i)}$ that 
$\mathcal{X}^{i,n,x}_0 = x$ and
\begin{equation}
\label{flow_statement:settingb}
  \mathcal{X}^{i,n,x}_k
=
  \mathcal{X}^{i,n,x}_{k-1} + \tfrac{T}{n} f_i(\mathcal{X}^{i,n,x}_{k-1})
=
  \mathcal{E}^{(n)}_i(\mathcal{X}^{i,n,x}_{k-1})
=
  \underbrace{\mathcal{E}^{(n)}_i \circ \ldots \circ \mathcal{E}^{(n)}_i }_{k\text{ - times}}(x),
\end{equation}
let $A = \{(i,n,k) \in I \times \N \times \N \colon k \leq n+1 \}$ and
$B = \{(i,n,k) \in I \times \N \times \N \colon k \leq n \}$,
let 
	$(L_{i}^{n,k})_{(i,n,k) \in A} \subseteq [0, \infty)$, 
	$\mathcal{G} = (\mathcal{G}_{i, \varepsilon}^{n,k})_{(i,n,k,\varepsilon) \in B \times (0,1]} \subseteq (0, \infty]$,
	$\mathcal{H} = (\mathcal{H}_{i, \varepsilon}^{n,k})_{(i,n,k,\varepsilon) \in B \times (0,1]} \subseteq (0, \infty]$
satisfy for all 
	$(i,n,k) \in B$, 
	$\varepsilon \in (0,1]$ 
that
\begin{equation}
\label{flow_statement:eq5a}
	L_{i}^{n,n+1} = c, \qquad
	L_{i}^{n,k} = 1 + \frac{cT}{n} = 
    \mathcal{G}_{i, \varepsilon}^{n,k},
\qandq 
	\mathcal{H}_{i, \varepsilon}^{n,k} = H_{i, \frac{\varepsilon}{\max\{T, 1\}}},
\end{equation}
and let $Q_1, Q_2, \ldots, Q_N \in [0, \infty)$ satisfy for all $i \in \{1, 2, \ldots, N\}$ that $Q_i = 2r_i+8\iota_i+\frac{2r_0\beta_i}{1-\rho}$.
Combining \cref{flow_statement:ass1,flow_statement:setting1,flow_statement:eq5a} with
\eqref{approximation_spaces_GH:eq1} and \cref{euler_AS} assures that
\begin{equation}
\label{flow_statement:7}
    \left(B \ni (i,n,k) \mapsto (\mathcal{E}_i^{(n)}) \in C(\R^{\inputStruc(i)}, \R^{\inputStruc(i)}) \right) \in 
    \ASGH
    _{a, \mathfrak{w}, \mathcal{G}, \mathcal{H}}
    ^{r_0,  r_1+4\iota_1, \ldots,  r_N + 4\iota_N, 0, 0}
    (B, (\mathfrak{d} \proddimmap \id_{\N} \proddimmap \id_{\N})|_B)
\end{equation}\cfload.
Next note that \eqref{flow_statement:ass1} and \eqref{flow_statement:setting1} ensure
that for all $n \in \N$, $k \in \{1, 2, \ldots, n \}$, $i \in I$, $x, y \in \R^{\inputStruc(i)}$ it holds that
\begin{equation}
\label{flow_statement:eq1}
\begin{split}
  \norm{ \mathcal{E}^{(n)}_i(x)  - \mathcal{E}^{(n)}_i(y) }
&= 
  \norm{(x + \tfrac{T}{n} f_i(x))  -  (y + \tfrac{T}{n} f_i(y) ) } \\
&\leq
  \norm{x-y} + \tfrac{T}{n}\norm{ f_i(x)  -  f_i(y)  }
\leq
  (1 + \tfrac{cT}{n})\norm{x-y}
=
		L_{i}^{n,k} \norm{x-y}.
\end{split}
\end{equation}
Furthermore, observe that \eqref{flow_statement:ass1}, \eqref{flow_statement:eq5a}, and the fact that for all $n \in \N$, $r \in (0, \infty)$ it holds that $(1 + \frac{r}{n})^{n} < e^r$ assure that
\begin{equation}
\label{flow_statement:eq6}
\begin{split}
    &\sup_{(i,n,k) \in A, \varepsilon \in (0,1]}
    \frac{
      \big[\prod_{l = k+1}^{n+1}  L_{i}^{n,l} \big]
       \left[
          \max \left(\{\mathcal{H}^{n,l}_{i, \varepsilon} \in (0, \infty] \colon l \in  \N \cap (0,k)\} \cup \{ 1 \} \right)
          \big[ \prod_{l = 1}^{k-1} \max\{\mathcal{G}^{n,l}_{i, \varepsilon},1 \} \big]
        \right]^\kappa
    }{
      \varepsilon^{-\rho}
      \big[ \smallprod{l}{1}{N} |\mathfrak{d}_l(i)|^{\beta_l} \big]
      n^0 }
    \\ & = [\max\{T, 1\}]^\rho
    \sup_{(i,n,k) \in A, \varepsilon \in (0,1]}
    \frac{
      c
      (1 + \frac{cT}{n})^{\max\{n-k-1,  0\}}
      (1 + \frac{cT}{n})^{\kappa (k-1)}
      \max \{ |H_{i, \frac{\varepsilon}{\max\{T, 1\}}}|^\kappa, 1 \}
    }{
      \big(\frac{\varepsilon}{\max\{T, 1\}}\big)^{\!-\rho}
      \smallprod{l}{1}{N} |\mathfrak{d}_l(i)|^{\beta_l} 
    }
    \\ & \leq [\max\{T, 1\}]^\rho
    \sup_{(i,n,k) \in A, \varepsilon \in (0,1]}
    \frac{
      c
      (1 + \frac{cT}{n})^{n}
      (1 + \frac{cT}{n})^{\kappa n}
      \max \{|H_{i, \frac{\varepsilon}{\max\{T, 1\}}}|^\kappa, 1\}
    }{
      \big(\frac{\varepsilon}{\max\{T, 1\}}\big)^{\!-\rho}
      \smallprod{l}{1}{N} |\mathfrak{d}_l(i)|^{\beta_l} 
    }
    \\ & \leq
    c [\max\{T, 1\}]^\rho  e^{cT(1 + \kappa)}
    \sup_{i \in I, \varepsilon \in (0,1]}
    \frac{
      \max\{|H_{i,\varepsilon}|^\kappa, 1\}
    }{
      \varepsilon^{-\rho}
      \smallprod{l}{1}{N} |\mathfrak{d}_l(i)|^{\beta_l} 
    } < \infty.
\end{split}
\end{equation}
Moreover, observe that the assumption that
$g 
    \in 
    \ASrates{a, \mathfrak{w}}{
    r_0, r_1, \ldots, r_N
    }(I, \mathfrak{d})$
implies that 
\begin{equation}
\begin{split}
    \left(I \times \N \ni (i,n) \mapsto (g_i) \in C(\R^{\inputStruc(i)}, \R) \right) 
    &
    \in 
    \ASrates{a, \mathfrak{w}}
    {r_0,  r_1+4\iota_1, \ldots,  r_N + 4\iota_N, 0}
    (I \times \N, \mathfrak{d} \proddimmap \id_{\N}).
\end{split}
\end{equation}
Combining this,
\eqref{flow_statement:ass1}, 
\eqref{flow_statement:7},
\eqref{flow_statement:eq1},
and
\eqref{flow_statement:eq6}
with
\cref{composition_statement}
establishes that
\begin{equation}
\label{flow_statement:eq07}
    \Big( \prodofstruct\limits_{n \in \N} 
    \big( g \compofstruct \underbrace{\mathcal{E}^{(n)} \compofstruct \ldots \compofstruct \mathcal{E}^{(n)} }_{n \text{ - times}}\big) \Big)
\in
  \ASrates{a, \mathfrak{w}}{
  \frac{2r_0}{1-\rho}, Q_1, Q_2,\ldots, Q_N,
  1 + \frac{2r_0(\kappa+1)}{1-\rho}
  }(I \times \N, \mathfrak{d} \proddimmap \id_{\N})
\end{equation}\cfload.
Next observe that \eqref{flow_operator:eq1}, \eqref{flow_operation:eq1}, \eqref{flow_statement:ass1}, \eqref{flow_statement:settinga}, \eqref{flow_statement:settingb}, and \cref{Euler_method_bound} assure that for all $i \in I$, $n \in \N$, $x \in \R^{\inputStruc(i)}$ it holds that
\begin{equation}
\begin{split}
  &\Big|
    (g \flowstruct_T f)_i(x) 
    - 
    \big( 
      \prodofstruct\limits_{k \in \N }  
        \big( 
          g \compofstruct \overbrace{\mathcal{E}^{(k)} \compofstruct \ldots \compofstruct \mathcal{E}^{(k)} }^{k \text{ - times}}
        \big)
    \big)_{(i,n)}
    (x)
  \Big| 
 =
  \left|
    g_i(X^{i,x}_T)
    - 
    g_i(\mathcal{X}^{i,n,x}_n)
    \right| 
    \leq
    c \Norm{X^{i,x}_T - \mathcal{X}^{i,n,x}_n} \\
    &\leq \frac{c^2   \max\{c, 1\}  T^2 (T+1)e^{2cT} \max\{\norm{f_i(0)}, 1\}}{n} 
    (1 + \norm{x}).
\end{split}
\end{equation}
This, 
\eqref{flow_statement:ass1}, and
the assumption that for all $v \in [0,\infty)$ it holds that
	$\mathfrak{w}(v) = (1 + v^\kappa)^{-1}$
ensure that 
\begin{equation}
\begin{split}
    & \sup_{i \in I} \sup_{n \in \N} \sup_{x \in \R^{\inputStruc(i)}} 
    \frac{ 
      \mathfrak{w}( \norm{ x })
      \Big| 
        (g \flowstruct_T f)_i(x) -   \big( 
        \prodofstruct\limits_{k \in \N }  
          \big( 
            g \compofstruct \overbrace{\mathcal{E}^{(k)} \compofstruct \ldots \compofstruct \mathcal{E}^{(k)} }^{k \text{ - times}}
          \big)
        \big)_{(i,n)}
        (x) 
      \Big|
    }{
      n^{-1}
      \smallprod{l}{1}{N} |\mathfrak{d}_l(i)|^{\alpha_l} 
    }
    \\ & \leq
    \sup_{i \in I} \sup_{n \in \N} \sup_{x \in \R^{\inputStruc(i)}} 
    \left(
    \left[\frac{1 + \norm{x}}{1 + \norm{x}^\kappa} \right]
    c^2 \max\{c, 1\}  T^2 (T+1)e^{2cT} 
     \left[\frac{\max \{ \norm{f_i(0)}, 1 \}}{
    	 \smallprod{l}{1}{N} |\mathfrak{d}_l(i)|^{\alpha_l} 
        }\right]\right)
    \\ & \leq c^2 \max\{c, 1\}  T^2 (T+1)e^{2cT} \sup_{i \in I} \frac{\max \{ \norm{f_i(0)}, 1 \}}{
    \smallprod{l}{1}{N} |\mathfrak{d}_l(i)|^{\alpha_l} 
    } < \infty.
\end{split}
\end{equation}
Combining this and \eqref{flow_statement:eq07} with \cref{lim_approximation} 
establishes that
$
    (g \flowstruct_T f) \in 
    \ASrates{a, \mathfrak{w}}
    {R_0 ,R_1,R_2, \ldots,R_N}
    (I , \mathfrak{d})
$ \cfload.
The proof of \cref{flow_statement} is thus complete.
\end{proof}

\subsection{Approximation of PDE flows with rectified neural networks}
\label{subsect:PDE_flows_relu}

\newcommand{\rect}{ {\cfadd{Def:rectifier_function} \mathfrak{r}} }
\begin{definition}[Rectifier function]
\label{Def:rectifier_function}
We denote by $\rect \colon \R \to \R$ the function which satisfies for all $x\in \R$ that
$
    \rect(x) = \max\{x,0\}
$
and we call $\rect$ the rectifier function.%
\end{definition}

\cfclear
\begin{lemma}[Representations of the identity with ReLU ANNs]
\label{identity_representation}
There exists $\ANNi = (\ANNi_d)_{d \in \N} \subseteq \ANNs$ such that for all $d \in \N$ it holds
\begin{enumerate}[(i)]
    \item \label{identity_representation:1} that $\dims(\ANNi_d)= (d, 2d, d)$ and
    \item \label{identity_representation:2} that $\realisation_\rect(\ANNi_d) = \id_{\R^d}$
\end{enumerate}
\cfload.
\end{lemma}

\begin{proof}%
This statement was proven in \cite[Lemma 5.5]{JentzenSalimovaWelti2021}.
\end{proof}

\cfclear
\begin{cor}[PDE flow in approximation spaces with ReLU activation]
\label{flow_statement_ReLu}
Let $c, T, r_0, r_1, \allowbreak \alpha, \allowbreak \beta \in [0,\infty)$, $\rho \in [0,1)$, $\kappa \in [1,\infty)$, 
$H = (H_{d, \varepsilon})_{(d, \varepsilon) \in \N \times (0,1]} \subseteq (0, \infty]$ satisfy
$
    \sup_{d \in \N} \sup_{\varepsilon \in (0,1]}
    \frac{
    |H_{d, \varepsilon}|^\kappa
    }{ \varepsilon^{-\rho}
    d^{\beta}
    }  < \infty
$,
let $\mathfrak{w} \colon [0, \infty) \to [0, \infty)$ satisfy
for all $v \in [0,\infty)$ that 
$
  \mathfrak{w}(v) = (1 + v^\kappa)^{-1}
$,
and let 
$
    f = (f_d)_{d \in \N}
    \in 
    \ASGH_{\rect, \mathfrak{w}, c, H}
    ^{r_0, r_1}(\N, \id_\N)
$,
$
    g = (g_d)_{d \in \N}
    \in 
    \ASrates{\rect, \mathfrak{w}}{
    r_0, r_1
    }(\N, \id_\N)
$
satisfy for all $d \in \N$, $x, y \in \R^{d}$ that 
$f_d \in C( \R^{d}, \R^{d})$, 
$g_d \in C( \R^{d}, \R)$, $\norm{f_{d}(0)} \leq cd^\alpha$, and
\begin{equation}
\label{flow_statement_ReLu:ass1}
  \max \{ 
    \norm{ f_d(x) - f_d(y) } 
    , 
    \left| g_d(x) - g_d(y) \right|  
  \} 
\leq
  c\norm{ x - y}
\end{equation}
\cfload.
Then
\begin{equation}
\label{flow_statement_ReLu:concl1}
    (g \flowstruct_T f) 
\in
    \ASrates{\rect, \mathfrak{w}}
    {1 + \frac{2r_0 (\kappa + 2)}{1 - \rho},  8 + \alpha + 2r_1 + \frac{2r_0(\alpha(\kappa+1)+\beta)}{1-\rho}}
    (\N, \id_\N)
\end{equation}
\cfout.
\end{cor}

\begin{proof}%
Combining \cref{flow_statement} 
with \cref{identity_representation}
establishes \eqref{flow_statement_ReLu:concl1}. The proof of \cref{flow_statement_ReLu} is thus complete.
\end{proof}

\cfclear
\begin{cor}[Approximations of PDE flows with ReLU ANNs]
\label{relu_flow_statement}
Let $c, T, r_0, r_1, \alpha, \beta \in [0,\infty)$,
$\kappa \in [1, \infty)$,  $\rho \in [0,\nicefrac{1}{\kappa})$, for every $d \in \N$ let $f_{d} \colon \R^d \to \R^d$ and $g_{d} \colon \R^d \to \R$ satisfy for all $x,y \in \R^d$ that $\norm{f_{d}(0)} \leq cd^\alpha$ and
\begin{equation}
\label{relu_flow_statement:ass1}
    \max\{\norm{ f_{d}(x) - f_{d}(y) }, |g_{d}(x) - g_{d}(y)|\}
    \leq c\norm{ x - y },
\end{equation}
and let $(\ANNf_{ d, \varepsilon})_{( d, \varepsilon) \in \N \times (0, 1]} \subseteq \ANNs$, $(\ANNg_{ d, \varepsilon})_{( d, \varepsilon) \in \N \times (0, 1]} \subseteq \ANNs$ satisfy for all $d \in \N$, $\varepsilon \in (0, 1]$, $x \in \R^d$ that
\begin{equation}
\label{relu_flow_statement:ass4a}
    \max\{\param(\ANNf_{d, \varepsilon}), \param(\ANNg_{d, \varepsilon}) \} \leq c \varepsilon^{-r_0} d^{r_1}, \qquad
    \realisation_\rect(\ANNf_{d, \varepsilon}) \in C(\R^d, \R^d ), \qquad \realisation_\rect(\ANNg_{d, \varepsilon}) \in C(\R^d, \R ),
\end{equation}
\begin{equation}
    \label{relu_flow_statement:ass4}
    \max \left\{\norm{ f_{d}(x) - \big(\realisation_\rect(\ANNf_{d, \varepsilon})\big)(x)},
    |g_{d}(x) - \big(\realisation_\rect(\ANNg_{d, \varepsilon})\big)(x)| \right\} \leq \varepsilon \left( 1 + \norm{x}^\kappa \right),
\end{equation}
\begin{equation}
\label{relu_flow_statement:ass5}
    \andq
    \norm{ \big( \realisation_\rect(\ANNf_{d, \varepsilon}) \big) (x) }
    \leq    c(\varepsilon^{-\rho}d^\beta + \norm{x})
\end{equation}
\cfload.
Then there exist 
	$K \in [0, \infty)$, $\ANNu = (\ANNu_{d,\varepsilon})_{(d, \varepsilon) \in \N \times (0,1]} \subseteq \ANNs$ 
such that for all 
	$d \in \N$, 
	$\varepsilon \in (0, 1]$, 
	$x \in \R^d$ 
it holds that 
	$\realisation_\rect(\ANNu_{d,\varepsilon}) \in C(\R^d, \R)$, 
	$\paramANN(\ANNu_{d,\varepsilon}) \leq K \varepsilon^{-1 - \frac{2r_0 (\kappa + 2)}{1 - \rho\kappa} }  d^{8 + \alpha + 2r_1 + \frac{2r_0 (\alpha(\kappa+1)+\beta \kappa)}{1-\rho\kappa}}$, and
\begin{equation}
\label{relu_flow_statement:concl}
    |\big(\Flow_{f_d}^T (g_d)\big)(x) - \big(\realisation_\rect(\ANNu_{d,\varepsilon})\big)(x)|
    \leq \varepsilon (1 + \norm{x}^\kappa)
\end{equation}
\cfout.
\end{cor}

\begin{proof}%
Throughout this proof assume w.l.o.g.\ that $c > 0$, 
let
$\mathfrak{w}\colon[0, \infty) \to [0, \infty)$ satisfy for all $v \in [0, \infty)$ that 
$
  \mathfrak{w}(v) = (1 + v^\kappa)^{-1}
$, and 
let 
$H = (H_{d, \varepsilon})_{(d, \varepsilon) \in \N \times (0,1]} \subseteq (0, \infty]$ 
satisfy for all 
	$d \in \N$, $\varepsilon \in (0, 1]$
that
$
    H_{d, \varepsilon} = \varepsilon^{-\rho}d^\beta.
$
Observe that
\begin{equation}
\label{relu_flow_statement:3}
\begin{split} 
	 \sup_{d \in \N} \sup_{ \varepsilon \in (0,1]}
	 \frac{
	 |H_{d, \varepsilon}|^\kappa
	 }{ \varepsilon^{-\rho\kappa}
	 d^{\beta\kappa}
	 } 
	= 1 < \infty.
\end{split}
\end{equation}
Next note that \eqref{relu_flow_statement:ass4a}, \eqref{relu_flow_statement:ass4}, and \cref{approximation_space_equivalence}
assure that
\begin{equation}
\label{relu_flow_statement:4}
    \left(\N \ni d \mapsto (g_d) \in C(\R^d, \R) \right)
    \in \ASrates{\rect, \mathfrak{w}}{r_0, r_1}(\N, \id_\N)
\end{equation}
\cfload.
Furthermore, note that \eqref{relu_flow_statement:ass4a}, \eqref{relu_flow_statement:ass4}, \eqref{relu_flow_statement:ass5}, and \cref{approximation_space_equivalenceGH} imply that
\begin{equation}
\label{relu_flow_statement:5}
    \left(\N \ni d \mapsto (f_d) \in C(\R^d, \R^d) \right)
    \in \ASGH_{\rect,\mathfrak{w}, c, H}^{r_0, r_1}(\N, \id_\N).
\end{equation}
Combining 
	this, 
	\eqref{relu_flow_statement:ass1}, 
	\eqref{relu_flow_statement:3}, and 
	\eqref{relu_flow_statement:4} 
with 
\cref{flow_statement_ReLu} 
implies that
\begin{equation}
 \label{relu_flow_statement:6}
    \big(\Flow_{f_d}^T (g_d)\big)_{d \in \N} = (g \flowstruct_T f)
\in 
	\ASrates{\rect, \mathfrak{w}}{1 + \frac{2r_0 (\kappa + 2)}{1 - \rho\kappa},  8 + \alpha + 2r_1 + \frac{2r_0 (\alpha(\kappa+1)+\beta \kappa)}{1-\rho\kappa}}(\N,\id_\N)
\end{equation}
\cfload.
This and \cref{approximation_space_equivalence} establish \eqref{relu_flow_statement:concl}.
The proof of \cref{relu_flow_statement} is thus complete.
\end{proof}

\cfclear
\begin{cor}[Approximations of first order PDEs with ReLU ANNs]
\label{relu_flow_statement_max}
Let $c, T \in [0,\infty)$, 
$\kappa \in [1, \infty)$, $\rho \in [0,\nicefrac{1}{\kappa})$, 
for every $d \in \N$ let $f_{d} \in C^1(\R^d , \R^d)$, $g_{d} \in C^1( \R^d , \R)$ satisfy for all $x,y \in \R^d$ that $\norm{f_{d}(0)} \leq cd^c$ and
\begin{equation}
\label{relu_flow_statement_max:ass1}
    \max\{\norm{ f_{d}(x) - f_{d}(y) }, |g_{d}(x) - g_{d}(y)|\}
    \leq c \norm{ x - y },
\end{equation}
and let $(\ANNf_{ d, \varepsilon})_{( d, \varepsilon) \in \N \times (0, 1]} \subseteq \ANNs$, $(\ANNg_{ d, \varepsilon})_{( d, \varepsilon) \in \N \times (0, 1]} \subseteq \ANNs$ satisfy for all $d \in \N$, $\varepsilon \in (0, 1]$, $x \in \R^d$ that
\begin{equation}
\label{relu_flow_statement_max:ass4a}
    \max\{\param(\ANNf_{d, \varepsilon}), \param(\ANNg_{d, \varepsilon}) \} \leq c \varepsilon^{-c} d^{c}, \qquad
    \realisation_\rect(\ANNf_{d, \varepsilon}) \in C(\R^d, \R^d ), \qquad \realisation_\rect(\ANNg_{d, \varepsilon}) \in C(\R^d, \R ),
\end{equation}
\begin{equation}
    \label{relu_flow_statement_max:ass4}
    \max \left\{\norm{ f_{d}(x) - \big(\realisation_\rect(\ANNf_{d, \varepsilon})\big)(x)},
    |g_{d}(x) - \big(\realisation_\rect(\ANNg_{d, \varepsilon})\big)(x)| \right\} \leq \varepsilon \left( 1 + \norm{x}^\kappa \right),
\end{equation}
\begin{equation}
\label{relu_flow_statement_max:ass5}
    \andq
    \norm{ \big( \realisation_\rect(\ANNf_{d, \varepsilon}) \big) (x) }
    \leq    c(\varepsilon^{-\rho}d^c + \norm{x})
\end{equation}
\cfload.
Then
\begin{enumerate}[(i)]
\item \label{relu_flow_statement_max:concl1}
there exist unique $u_d \in C^1([0, T] \times \R^d, \R)$, $d \in \N$, such that for all 
    $d \in \N$,
	$t \in [0, T]$,
	$x \in \R^d$
it holds that 
$u_d(0,x) = g_d(x)$ and
\begin{equation}
 \tfrac{\partial  u_d}{\partial t}  (t,x)  
=
  \tfrac{\partial u_d}{\partial x} (t,x)\, f_d(x)
\end{equation}
and
\item \label{relu_flow_statement_max:concl2}
there exist $K \in [0, \infty)$, $\ANNu = (\ANNu_{d,\varepsilon})_{(d, \varepsilon) \in \N \times (0,1]} \subseteq \ANNs$ such that for all $d \in \N$, $\varepsilon \in (0, 1]$ it holds that $\realisation_\rect(\ANNu_{d,\varepsilon}) \in C(\R^d, \R)$, $\paramANN(\ANNu_{d,\varepsilon}) \leq K \varepsilon^{-K}  d^{K}$, and
\begin{equation}
    \sup_{x \in \R^d} \frac{
    |u_d(T,x) - \big(\realisation_\rect(\ANNu_{d,\varepsilon})\big)(x)|}
    {1 +\norm{x}^\kappa}
    \leq \varepsilon.
\end{equation}
\end{enumerate}
\end{cor}

\begin{proof}%
\cref{relu_flow_statement_max} is a consequence of \cref{first_order_FC} and  \cref{relu_flow_statement}.
See \cite[Corollary 4.11]{Beneventano2020v1} for more details.
\end{proof}

\section*{Acknowledgments}
This project has been partially funded by the SNSF-Research project 200020{\_}175699 ``Higher order numerical approximation methods for stochastic partial differential equations''.
This work has also been partially funded by the National Science Foundation of China (NSFC) under grant number 12250610192.
Moreover, we gratefully acknowledge the Cluster of Excellence EXC2044-390685587, Mathematics M\"unster: Dynamics-Geometry-Structure funded by the Deutsche Forschungsgemeinschaft (DFG, German Research Foundation).

\bibliographystyle{acm}
\bibliography{../1_Main_bibfile/Main_bibfile.bib}

\end{document}